\documentclass[10pt,a4paper]{amsart}

\usepackage{graphicx,amssymb,color,amscd}
\usepackage[all]{xy}
\SelectTips{cm}{}

\newtheorem{thm}{Theorem}

\newtheorem{cor}{Corollary}

\theoremstyle{plain}
\newtheorem{theorem}{Theorem}[section]
\newtheorem*{theorem*}{Theorem}
\newtheorem{lemma}[theorem]{Lemma}
\newtheorem{proposition}[theorem]{Proposition}
\newtheorem{corollary}[theorem]{Corollary}
\newtheorem*{corollary*}{Corollary}
\theoremstyle{definition}

\theoremstyle{remark}
 
\newtheorem{remark}[theorem]{Remark}
\newtheorem{example}[theorem]{Example}

\numberwithin{equation}{section}
\numberwithin{figure}{section}

\def\tr{\mathop{\mathrm{tr}}\nolimits}

\newcommand{\bd}{\begin{description}}   
\newcommand{\ed}{\end{description}} 
\newcommand{\ba}{\begin{array}}      \newcommand{\ea}{\end{array}} 
\newcommand{\bc}{\begin{center}}     \newcommand{\ec}{\end{center}} 
\newcommand{\be}{\begin{enumerate}}  \newcommand{\ee}{\end{enumerate}} 
\newcommand{\beq}{\begin{eqnarray}}  \newcommand{\eeq}{\end{eqnarray}} 
\newcommand{\beQ}{\begin{eqnarray*}} \newcommand{\eeQ}{\end{eqnarray*}} 
\newcommand{\F}{\textrm{F}_l}

\def\co{\colon\thinspace}
\def\id{\mathop{\mathrm{id}}\nolimits}
\def\coef{\mathrm{coeff}_{\hbar}}
\newcommand{\oh}[1]{\pmod{\hbar^{#1}}}

\begin{document} 
\title[Universal $sl_2$ invariant and Milnor invariants]{The universal $sl_2$ invariant and Milnor invariants} 

\author[J.B. Meilhan]{Jean-Baptiste Meilhan} 
\address{Univ. Grenoble Alpes, Institut Fourier, F-38000 Grenoble, France}
         \email{jean-baptiste.meilhan@ujf-grenoble.fr}
\author[S. Suzuki]{Sakie Suzuki} 
\address{The Hakubi Center for Advanced Research/Research Institute for Mathematical Sciences, Kyoto University, Kyoto, 606-8502, Japan. }
         \email{sakie@kurims.kyoto-u.ac.jp}

\begin{abstract}
The universal $sl_2$ invariant of string links has a universality property for the colored Jones polynomial of links, 
and takes values in the $\hbar$-adic completed tensor powers of the quantized enveloping algebra of $sl_2$.
In this paper, we exhibit explicit relationships between the universal $sl_2$ invariant and Milnor invariants, 
which are classical invariants generalizing the linking number, providing some new topological insight into quantum invariants.
More precisely, we define a reduction of the universal $sl_2$ invariant, and show how it is captured by Milnor concordance invariants.
We also show how a stronger reduction corresponds to Milnor link-homotopy invariants.
As a byproduct, we give explicit criterions for invariance under concordance and link-homotopy of the universal $sl_2$ invariant, and in particular for sliceness.
Our results also provide partial constructions for the still-unknown weight system of the universal $sl_2$ invariant. 
\end{abstract}



\maketitle

\section{Introduction}
The theory of quantum invariants of knots and links emerged in the middle of the eighties, after the fundamental work of V. F. R. Jones. 
Instead of the classical tools of topology, such as algebraic topology, used until then, this new class of invariants was derived from interactions of knot theory with other fields of mathematics, such as operator algebras and representation of quantum groups, and revealed close relationships with theoretical physics. 
Although this gave rise to a whole new class of powerful tools in knot theory, we still lack a proper understanding of the topological information carried by quantum invariants. 
One way to attack this fundamental question is to exhibit explicit relationships with classical link invariants. 
The purpose of this paper is to give such a relation, by showing how a certain reduction of the universal $sl_2$ invariant is captured by Milnor invariants. 

\textit{Milnor invariants} were originally defined by J. Milnor for links in $S^3$ \cite{Milnor,Milnor2}. 
Their definition contains an intricate indeterminacy, which was shown by N. Habegger and X. S. Lin 
to be equivalent to the indeterminacy in representing a link as the closure of a \textit{string link}, i.e. of a pure tangle without closed components \cite{HL1}.
Milnor invariants are actually well-defined integer-valued invariants of framed string links,  
and the first non-vanishing Milnor string link invariants can be assembled into a single \textit{Milnor map} $\mu_k$. 
See Section \ref{2} for a review of Milnor string link invariants. 

Milnor invariants constitute an important family of classical (string) link invariants, 
and as such, their connection with quantum invariants has already been the subject of several works. 
The first attempt seems to be due to L. Rozansky, who conjectured a formula relating Milnor invariants to the Jones polynomial \cite{rozansky}. 
An important step was taken by Habegger and G. Masbaum, who showed explicitly in \cite{HMa} how Milnor invariants are related to the Kontsevich integral. 
More recently, A. Yasuhara and the first author gave explicit formulas relating Milnor invariants to the HOMFLYPT polynomial \cite{MY}. 

The \textit{universal $sl_2$ invariant} $J(L)$ for an  $l$-component framed string link $L$ takes values in the $l$-fold completed tensor power $U_{\hbar}(sl_2)^{\hat \otimes  l}$ of the quantized enveloping algebra $U_{\hbar}(sl_2)$ of $sl_{2}$, and has the universal property for  the colored Jones polynomial \cite{Law1,Law2,O,RT}. 
See Section \ref{3} for the definitions of $U_{\hbar}(sl_2)$ and the universal $sl_2$ invariant.
The second author studied in \cite{sakie1,sakie2,sakie3} the universal $sl_2$ invariant of several classes of string links
satisfying vanishing properties for Milnor invariants.\footnote{More precisely, these results are for \textit{bottom tangles}; but we can identify bottom tangles with string links via a fixed one-to-one correspondence, see \cite{H1}. } 
In this paper, we further explore the relation with the universal $sl_2$ invariant and Milnor invariants.

Before we proceed with the description of our results, 
let us recall the relationship between the Kontsevich integral, the universal $sl_{2}$ invariant and the colored Jones polynomial. 
The Kontsevich integral $Z(L)$ for an $l$-component string links $L$ takes values in the completed space $\mathcal{A}(l)$ of Jacobi diagrams on the disjoint union of $l$ intervals.
For the closure link $\text{cl}(L)$ of $L$, the colored Jones polynomial $J_{V_1,\ldots,V_l}(\text{cl}(L))$, 
with a finite dimensional representation $V_i$ of $U_{\hbar}(sl_{2})$ attached to the $i$th component, 
takes values in $ \mathbb{Z}[q^{1/4},q^{-1/4}] \subset \mathbb{Q}[[\hbar]]$,  where $q=\exp \hbar$.
The Kontsevich integral has the universal property for finite type invariants, thus for quantum invariants. 
This implies that there exists a algebra homomorphism $W^{U}\co \mathrm{Im}(Z)  \to  U_{\hbar}(sl_2)^{\hat{\otimes} l}$, the so called \textit{weight system} for the universal $sl_{2}$ invariant, such that  $W^{U}\circ Z=J$. 
(This is well-known,  and 
follows from the fact that the coefficients of the universal invariant have a finite type property, in a strictly similar way as in \cite[Cor. 7.5]{ohtsuki}.)
However, no explicit formula for $W^{U}$ is known yet.
There is, however, a graded algebra homomorphism called  \textit{the universal $sl_2$ weight system}:
$$ W\co  {\mathcal{A}}(l) \to U(sl_2)^{\otimes l}[[\hbar]], $$ 
where $U(sl_2)$ is the universal enveloping algebra of $sl_2$. 
As a summary, we have the following commutative diagram:
\[\xymatrix@C=4mm{
 & \{l\text{-comp. string links}  \} \ar @{->}[dl]_{Z} \ar @{->}[d]^{J} \ar @{->}[r]^{\text{  cl} }& \{l\text{-component links} \}\ar @{->}[d]^{J_{V_1,\ldots,V_l }}   \\
{\mathcal{A}(l)  \ar @{->}[dr]_{W} \supset \mathrm{Im}(Z)}  \ar @{-->}[r]^{W^{U}} & U_{\hbar}(sl_2)^{\hat{\otimes} l}  \ar @{->}[r]^{\tr_q^{V_1,\ldots,V_l}} & \mathbb{Q}[[\hbar]]   \\
 & U(sl_2)^{\otimes l}[[\hbar]]\ar @{->}[ur]_{\tr_{\nu}^{V_1,\ldots,V_l}}  &  }\]
\noindent 
where $\tr_{q}^{V_1,\ldots,V_l}$ and $\tr_{\nu}^{V_1,\ldots,V_l}$ are variants of the quantum trace map. 
See \cite[Sec. 5]{H2} for the commutativity of the upper right square and  
\cite[Sec. 10]{LM} for  the boundary pentagon.
Note that the composition $\tr_{\nu}^{V_1,\ldots,V_l}\circ W=\tr_{q}^{V_1,\ldots,V_l}\circ W^{U}$ is the weight system for the colored Jones polynomial.

Note that the algebras $U_{\hbar}(sl_2)^{\hat\otimes l}$  and $U(sl_2)^{\otimes l}[[\hbar]]$ are isomorphic theoretically, but again, no explicit 
 isomorphism is known, see \cite{Kassel}. In this paper, we will fix a $\mathbb{Q}[[\hbar]]$-linear isomorphism
\begin{align}\label{qm}
\rho: U_{\hbar}(sl_2)^{\hat\otimes l} &\to U(sl_2)^{\otimes l}[[\hbar]], 
\end{align} 
with respect to the PBW basis  (see Section \ref{5.1}), so that we can compare the two \emph{different} $\mathbb{Q}$-linear maps $\rho \circ W^{U}$ and $W$.

Now, as mentioned above, Habegger and Masbaum showed in \cite{HMa} that, for an $l$-component string link $L$ with vanishing Milnor invariants of length $\le m$, we have
  \begin{align} \label{eq:HM}
  Z^t(L) = 1 + \mu_{m}(L) + (\textrm{terms of degree $\ge m+1$}) \in  \mathcal{A}^t(l), 
  \end{align}
where $Z^t$ is the projection of the Kontsevich integral onto the so--called ``tree part'' $\mathcal{A}^t(l)$ of  $\mathcal{A}(l)$,\footnote{The ``tree part'' is well-defined in the space $\mathcal{B}(l)$ of labeled Jacobi diagrams, which is isomorphic to $\mathcal{A}(l)$  as  a  $\mathbb{Q}$-module via an analogue of the PBW isomorphism from $S(sl_{2})$ to $U(sl_{2})$.   In this paper we will use $\mathcal{B}(l)$ rather than $\mathcal{A}(l)$, see Section \ref{sec:HM} for details. }  and where $\mu_{m}(L)$ is the Milnor map of $L$ regarded as an element of $\mathcal{A}^t(l)$. 
If we knew the weight system $W^{U}$ explicitly, we could easily deduce a relation between Milnor invariants and the universal $sl_{2}$ invariant by transfering Habegger-Masbaum's result (\ref{eq:HM}) via $W^{U}$; but, again, this is not the case. 
Actually, our first main result, Theorem \ref{sth2} below, implies that, 
when restricting to the image of Milnor map, we can identify $W$ and $W^{U}$ via the $\mathbb{Q}$-linear isomorphism $\rho$.  
In other words, we give a partial construction of $W^{U}$.

Let us now state the first main result explicitly. 
The Milnor map $\mu_{m}$ actually takes values in the space of tree Jacobi diagrams, i.e. connected and simply connected Jacobi diagrams. 
The restriction of $W$ to the space of tree Jacobi diagrams of degree $m$ takes values in $(U(sl_2)^{\otimes l})_{m+1} \hbar^m$, 
where $(U(sl_2)^{\otimes l})_{m+1}$ is the subspace of $U(sl_2)^{\otimes l}$ of homogeneous elements of degree $m+1$ with respect to the length of the words in $sl_2$ in the PBW basis, see Lemma \ref{wcc}.  
Thus,  if we  consider the projection
\begin{align*}
 \pi^t\co U(sl_2)^{\otimes l}[[\hbar]] &\to   \prod_{m\geq 1}(U(sl_2)^{\otimes l})_{m+1} \hbar^{m}, 
 \end{align*}
of $\mathbb{Q}$-modules  (see Section \ref{subseq}), then we can compare the maps $W \circ \mu_{m}$ and $J^t:= \pi^t\circ \rho \circ J$. We obtain the following result. 

\begin{thm}[Theorem \ref{sth2}]
Let $m\geq 1$.   
If $L$ is a string link with vanishing Milnor invariants of length $\le m$, then we have
$$
J^t(L)\equiv(W \circ \mu_{m}) (L) \oh{m+1}.
$$
\end{thm}
\noindent 
Here, and throughout the paper, we simply set $(\mathrm{mod} \ \hbar^{k})=( \mathrm{mod} \ \hbar^{k}U_{\hbar}^{\hat\otimes l})$ for $k\geq 1$ and  an appropriate $l\geq 1$. 

Theorem \ref{sth2} implies  a concordance-invariance property of $J^t$ as follows.
\begin{cor}[Corollary \ref{cor:main}]
Let $L, L' $  be two concordant string links with vanishing Milnor invariants of length $\le m$.
Then we have
 \begin{align*}
  J^t(L') \equiv J^t(L) \oh{m+1}. 
  \end{align*}
In particular, if $L$ is concordant to the trivial string link, then $J^t(L)$ is trivial.  
\end{cor}

There is also a variant of the Theorem \ref{sth2}, using another projection map $\tilde \pi^t$ onto a larger quotient of $U_{\hbar}^{\hat \otimes l}$; see Remark \ref{rem:tbw}. This provides another criterion for the universal $sl_2$ invariant, which applies in particular to slice, boundary or ribbon string links as follows.
\begin{thm}[Corollary \ref{cor:final}]
Let $L$ be an $l$-component string link with vanishing Milnor invariants. Then we have
 $$\rho(J(L))\in 1+ \prod_{ 1\leq i\leq j}(U(sl_2)^{\otimes l})_i\hbar^j. $$
\end{thm}
\noindent
This result strongly supports  Conjecture 1.5 in \cite{sakie2}, where the second author suggests that the universal $sl_2$ invariant of a bottom tangle with vanishing Milnor
invariants is contained in a certain subalgebra of $U_{\hbar}(sl_2)^{\hat \otimes l}$.

As emphasized above, Theorem \ref{sth2} is not a mere consequence of Habegger-Masbaum's work, 
and the proof will be given by comparing directly the definitions of the Milnor map and the universal $sl_{2}$ invariant. 
One of the main ingredients for the proof is a version for Milnor link-homotopy invariants. 
Recall that link-homotopy is the equivalence relation generated by self-crossing changes. 
Habegger and Lin showed that Milnor invariants indexed by sequences with no repetition form a complete set of link-homotopy invariants for string links \cite{HL1}. We can thus consider the link-homotopy reduction $\mu^h_{m}$ of the Milnor map $\mu_{m}$, see Section \ref{subseq}. 
On the other hand, we consider the projection of $\mathbb{Q}$-modules 
  $$  \pi^h\co U(sl_2)^{\otimes l}[[\hbar]]  \to \bigoplus_{m=1}^{l-1}\langle sl_2 \rangle _{m+1}^{(l)}\hbar^m, $$
where $\langle sl_2 \rangle_{m+1}^{(l)}\subset (U(sl_2)^{\otimes l})_{m+1}$ denotes the subspace spanned by tensor products such that each tensorand is of degree $\leq 1$, that is, roughly speaking, tensor products of $1$'s and elements of $sl_2$. 

It turns out that the restriction of the $sl_2$ weight system $W$ to the space of tree Jacobi diagrams with non-repeated labels takes values in this space $\bigoplus_{m=1}^{l-1}\langle sl_2 \rangle _{m+1}^{(l)}\hbar^m$. 
Thus, similarly as before, we can compare $W \circ \mu^h_{m}$ and $J^h:= \pi^h\circ \rho  \circ J$, and obtain the following second main result. 
\begin{thm}[Theorem \ref{sth2h}]
Let $m\geq 1$.   
If $L$ is  a string link with vanishing Milnor link-homotopy invariants of length $\le m$, then we have
$$
 J^h(L)\equiv (W \circ \mu^h_{m})(L) \oh{m+1}.  
$$
\end{thm}

Note that Theorem \ref{sth2h} cannot in general be simply deduced from Theorem  \ref{sth2} by a mere link-homotopy reduction process. 
(This is simply because a string link may in general have nonzero Milnor invariants of length $m$, yet vanishing Milnor link-homotopy invariants of length $m$.)
In order to prove Theorem \ref{sth2h}, one of the key results is  Proposition \ref{s241}, a link-homotopy invariance property for the map $J^h$. This reduces the proof to an explicit computation for a link-homotopy representative, given in terms of the lower central series of the pure braid group.  
In the process of proving Proposition \ref{s241}, we obtain, as in the case of Theorem \ref{sth2} and Remark \ref{rem:tbw}, a variant of Theorem \ref{sth2h} using another projection map, giving an algebraic criterion detecting link-homotopically trivial string links; see Remark \ref{rem:tbw2} and Corollary \ref{cor:final2}. 

It is worth mentioning here that the $sl_2$ weight system $W$ is not injective, and thus we do not expect that the universal $sl_2$ invariant detects Milnor invariants.   
This follows from the fact that $W$ takes values in the invariant part of $S(sl_2)^{\otimes l}[[\hbar]]$ and a simple argument comparing the dimensions of the domain and images. 
We will further study properties of the universal $sl_2$ weight system in a forthcoming paper \cite{MS}. 

The rest of the paper is organized as follows. 
In section 2, we review in detail the definition of Milnor numbers and of the Milnor maps $\mu_k$, and recall some of their properties. 
In section 3, we recall the definitions of the quantized enveloping algebra $U_{\hbar}(sl_2)$ and the universal $sl_2$ invariant, 
and recall how the framing and linking numbers are simply contained in the latter. 
Section 4 provides the diagrammatic settings for our paper; we review the definition of Jacobi diagrams, and their close relationships with the material from the previous sections. 
This allows us to give the precise statements of our main results in Section 5. 
Sections 6, 7 and 8 are dedicated to the proofs. Specifically, the link-homotopy version of our main result is shown in Section 6, while Section 7 contains the proof of the general case. 
Some of the key ingredients of these proofs require the theory of claspers, which we postponed to Section 8. 

\section{Milnor invariants}\label{2}

Throughout the paper, let $l\ge 1$ be some fixed integer.
 
Let $D^2$ denote the standard $2$-disk equipped with $l$ marked points  $p_1$, . . . , $p_l$ in its interior as shown in Figure \ref{fig:basis}. 
Fix also a point $e$ on the boundary of the disk $D^2$, 
and for each $i=1, . . ., l$, pick an oriented loop $\alpha_i$ in $D^2$ based at this $e$ and winding around $p_i$ in the trigonometric direction.  See Figure \ref{fig:basis}. 
\begin{figure}[!h]
 \centering
 \includegraphics[scale=0.6]{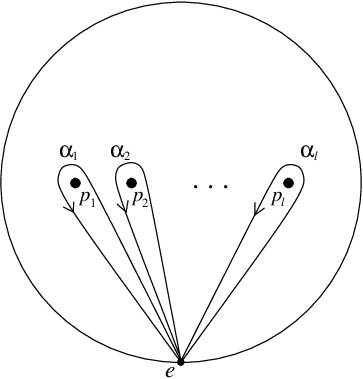}
 \caption{The disk $D^2$ with $l$ marked points $p_i$, and the arcs $\alpha_i$; $i=1, \cdots, l$. }\label{fig:basis}
\end{figure}

An $l$-component string link is a proper embedding of $l$ disjoint copies of the unit interval $[0,1]$ in $D^2\times [0,1]$, 
such that for each $i$, the image $L_i$ of the $i$th copy of $[0,1]$ runs from $(p_i,1)$ to $(p_i,0)$. The arc $L_i$ is called the $i$th component of $L$.  
An $l$-component string link is equipped with the \emph{downwards} orientation induced by the natural orientation of $[0,1]$.

In this paper, by a string link we will implicitly mean a \emph{framed} string link, that is, equipped with a trivialization of its normal tangent bundle. 
(Here, it is required that this trivialization agrees with the positive real direction at the boundary points.)
In the various figures of this paper, we make use of the blackboard framing convention.  

The ($0$-framed) $l$-component string link $\{p_1,...,p_l\}\times[0,1]$ in $D^2\times [0,1]$ is called 
the {\em trivial $l$-component string link} and is denoted by $\mathbf{1}_l$, 
or sometimes simply  $\mathbf{1}$ when the number of components is implicit.

Let $SL(l)$ denote the set of isotopy classes of $l$-component string links fixing the endpoints. 
The stacking product endows $SL(l)$ with a structure of monoid, with the trivial $l$-component string link $\mathbf{1}_l$ as unit element.   
In this paper, we use the notation $\cdot$ for the stacking product, with the convention that the rightmost factor is \emph{above}.
Note that the group of units of $SL(l)$ is precisely the pure braid group on $l$ strands $P(l)$ \cite{HL2}. 
\subsection{Artin representation and the Milnor map $\mu_k$ for string links} \label{sec:milnormap}
In this subsection we review Milnor invariants for string links, following \cite{HL1,HL2}. 

For an $l$-component string link $L=L_1\cup\cdots\cup L_l$ in $D^2\times [0,1]$, denote by 
$Y=(D^2\times [0,1])\setminus N(L)$ the exterior of an open tubular neighborhood $N(L)$ of $L$, 
and set $Y_0=(D^2\times \{0\})\setminus N(L)$ and $Y_1=(D^2\times \{1\})\setminus N(L)$. 
For $i=0,1$, the fundamental group of $Y_i$ based at $(e,i)$ identifies with the free group $\F$ on generators $\alpha_1,...,\alpha_l$. 

Recall that the lower central series of a group $G$ is defined inductively by $\Gamma_1G=G$ and $\Gamma_{k+1}G=[G,\Gamma_k G]$.  
By a theorem of J. Stallings \cite{stallings}, the inclusions 
$\iota_i:Y_i\longrightarrow Y$ induce isomorphisms 
$(\iota_i)_k:\pi_1(Y_t)/\Gamma_{k+1}\pi_1(Y_t) \longrightarrow \pi_1(Y)/\Gamma_{k+1}\pi_1(Y)$ for 
any positive integer $k$. 
Hence for each $k$, the string link $L$ induces an automorphism $(\iota_0)_k^{-1}\circ(\iota_1)_k$ of $\F/\Gamma_{k+1}\F$. 
Actually, this assignment defines a monoid homomorphism
$$ A_k: SL(l)\rightarrow \textrm{Aut}_0\left( \F/\Gamma_{k+1}\F \right), $$
called the \emph{$k$th Artin representation}, where $\textrm{Aut}_0\left( \F/\Gamma_{k+1}\F \right)$ denotes the group of automorphisms of $\F/\Gamma_{k+1}\F$ sending each generator $\alpha_j$ to a conjugate of itself and preserving the product $\prod_j \alpha_j$.  
More precisely, for each component $j$, consider the \emph{preferred $j$th longitude} of $L$, 
which is a $f_j$-framed parallel copy of $L_j$, 
where $f_j$ denotes the framing of component $j$. 
This defines an element $l_j$ in $\pi_1(Y)/\Gamma_{k+1}\pi_1(Y)$, and 
for any positive integer $k$, we set $l^k_j:=(\iota_0)_k^{-1}(l_j)\in \F/\Gamma_{k+1}\F$. 
Then we have that 
$A_k(L)$ maps each generator $\alpha_j$ to its conjugate 
$$ A_k(L):\alpha_j\mapsto l_j^k \alpha_j (l_j^k)^{-1}.$$
\noindent (Here, we denoted the image of $\alpha_j$ in the lower central series quotient $\F/\Gamma_{k+1}\F$ again by $\alpha_j$.) 

Denote by $SL_k(l)$ the set of $l$-component string links whose longitudes are all trivial in $\F/\Gamma_{k}\F$. 
We have a descending filtration of monoids
 $$ SL(l)=SL_1(l)\supset SL_2(l)\supset . . . \supset SL_k(l)\supset . . .$$
called the \emph{Milnor filtration},
and we can consider the map 
  $$ \mu_{k}: SL_k(l) \rightarrow \dfrac{\F}{\Gamma_2 \F}\otimes \dfrac{\Gamma_k \F}{\Gamma_{k+1} \F} $$
for each $k\ge 1$, which maps $L$ to the sum 
  $$ \mu_{k}(L) := \sum_{i=j}^l \alpha_j\otimes l^{k}_j, $$
called the \emph{degree $k$ Milnor map}. 

\subsection{Milnor numbers for string links} \label{milnornb}

As mentioned in the introduction, 
Milnor invariants were originally defined as numerical invariants.  Let us briefly review their definition and connection to the Milnor map. 

Let $\mathbb{Z}\langle \langle X_1, . . . ,X_l\rangle \rangle$ denote the ring of formal power series in the non-commutative variables $X_1,...,X_l$. 
The \emph{Magnus expansion} $E: \F\rightarrow \mathbb{Z}\langle \langle X_1, . . . ,X_l\rangle \rangle$ is 
the injective group homomorphism which maps each generator $\alpha_j$ of $\F$ to $1+X_j$ 
(and thus maps each $\alpha_j^{-1}$ to $1-X_j+X_j^2-X_j^3+\cdots$). 

Since the Magnus expansion $E$ maps $\Gamma_k\F$ to terms of degree $>k$, 
the coefficient $\mu_{i_1i_2...i_{m}j}(L)$ of $X_{i_1}\cdots X_{i_{m}}$ in the Magnus expansion $E(l^k_j)$ is 
a well-defined invariant of $L$ for any $m\leq k$,\footnote{Note that the integer $k$ can be chosen arbitrarily large, so this condition is not restrictive. } 
and it is called a \emph{Milnor $\mu$-invariant}, or Milnor number, of length $m+1$. 
Milnor invariants are sometimes referred to as higher order linking numbers, since $\mu_{ij}(L)$ is merely the linking number of component $i$ and $j$, 
while $\mu_{ii}(L)$ is just the framing of the $i$th component. 

For each $k\ge 1$, the $k$th term $SL_k(l)$  of the Milnor filtration 
coincides with the submonoid of $SL(l)$ of string links with vanishing Milnor $\mu$-invariants of length $\le k$, and  
the Milnor map $\mu_{k}$ is strictly equivalent to the collection of all Milnor $\mu$-invariants of length $k+1$.

Recall that two $l$-component string links $L$ and $L'$ are concordant if there is an embedding 
$$ f: \left(\sqcup_{i=1}^l [0,1]_i\right)\times I \longrightarrow \left(D^2\times I\right)\times I, $$  
where $\sqcup_{i=1}^l [0,1]_i$ is the disjoint union of $l$ copies of the unit interval $[0,1]$, 
such that $f\left( (\sqcup_{i=1}^l [0,1]_i)\times \{ 0 \} \right)=L\times \{ 0 \}$ and $f\left( (\sqcup_{i=1}^l [0,1]_i)\times \{ 1 \} \right)=L'\times \{ 1 \}$, and such that $f\left(\partial(\sqcup_{i=1}^l [0,1]_i)\times I\right)=(\partial L) \times I$.   
It is well known that Milnor numbers, hence Milnor maps, are not only isotopy invariants, but also concordance invariants : this is for example shown by Casson in \cite{casson}, although it is already implicit in Stallings' paper \cite{stallings}. 

\subsection{Link-homotopy and the lower central series of the pure braid group}\label{sec:pure} 

Recall that the link-homotopy is an equivalence relation on knotted objects generated by isotopies and self-crossing changes. 
Using the properties of Magnus expansion, Milnor proved that, if $I$ is a sequence \emph{with no repeated index}, 
then the corresponding invariant $\mu_I$ is a link-homotopy invariant, see Theorem 8 of  \cite{Milnor2}. 
Habegger and Lin subsequently proved that string links are classified up to link-homotopy by Milnor invariants with no repeated indices \cite{HL1}. 

More precisely, Habegger and Lin showed that the set 
 $ \bigcup_{m=2}^l \left\{  \mu_{I}\textrm{ $|$ }I\in \mathcal{I}_m \right\} $
forms a complete set of link-homotopy invariants for string links \cite{HL1,HL2}, where for each $m\in \{2, . . . ,l\}$, 
 $$ \mathcal{I}_m := \left\{\quad j_{\tau(1)}...j_{\tau(m-2)}j_{m-1}j_m \quad \left| \quad
                                       \begin{array}{c}
                                       1\le j_1<\cdots<j_{m-2}<j_{m-1}<j_m\le l \\
                                       \tau\in S_{m-2} 
                                      \end{array}\right.
\right\}. $$
In other words, $\mathcal{I}_m$ is the set of all sequences $j_1...j_m$ of $m$ non-repeating integers from $\{1,...,l\}$ such that $j_i<j_{m-1}<j_m$ for all $i\le m-2$. 

In this subsection, we use this result to give an explicit representative for the link-homotopy class of any string link in terms 
of the lower central series of the pure braid group.  

Recall that the pure braid group on $l$ strands $P(l)$ is generated by elements
  $$ A_{i,j}=\sigma_{j-1}\cdot . . . \cdot\sigma_{i+1}\cdot \sigma_{i}^2\cdot \sigma_{i+1}^{-1}\cdot . . . \cdot\sigma_{j-1}^{-1}\textrm{, for }1 \le i < j \le l,  $$
which may be represented geometrically as the pure braid where the $i$th string overpasses the strings ($i+1$), . . . , ($j-1$) and $j$, 
underpasses the $j$th string, then goes back to the $i$th position by overpassing all strings.
For convenience, we also define $A_{i,j}$ for $i>j$, by the convention $A_{i,j}:=A_{j,i}$.  

Given a sequence $J=j_1...j_m$  in $\mathcal{I}_m$, we define the pure braid
\begin{equation}\label{braidBJ}
  B_J^{(l)} = [[. . . [ [A_{j_1,j_2},A_{j_2,j_3}],A_{j_3,j_4}], . . . ] , A_{j_{m-1},j_m}],
\end{equation}
which lies in the $(m-1)$th term $\Gamma_{m-1} P(l)$ of the lower central series.
We simply write $B_J=B_J^{(l)}$ when there is no risk of confusion.

The pure braids $B_{J}$ ($J\in \mathcal{I}_m$) can be used to construct an explicit representative of the link-homotopy class of any string link as follows.
\begin{lemma}\label{lem:braidlh}
Any $l$-component string link $L$ is link-homotopic to $b^L_1\cdots b^L_{l-1}$, where 
\begin{equation}\label{eq:braid}
  b^L_i= \prod_{J\in \mathcal{I}_{i+1}} (B_{J})^{\mu_{J}(b^L_i)
  }\textrm{, where } 
  \mu_{J}(b^L_i)=\left\{\begin{array}{ll}
\mu_{J}(L)&\textrm{if $i=1$},\\ 
& \\
\mu_{J}(L)-\mu_{J}(b^L_1\cdots b^L_{i-1})& \textrm{if $i\geq 2$}.
\end{array}\right.
\end{equation}
\end{lemma}
\begin{remark}
 This lemma is to be compared with \cite[Thm. 4.3]{yasuhara} and \cite[Thm. 4.1]{MY}, where similar results are given in terms of tree claspers -- see Section \ref{8}. 
\end{remark}

\begin{proof}
In view of the link-homotopy classification result of Habegger and Lin recalled above, the lemma simply follows from a computation of Milnor invariants of  the pure braids $B_{J}$ ($J\in \mathcal{I}_m$). Specifically, using the additivity property of Milnor string link invariants (see e.g. \cite[Lem. 3.3]{MYpjm}), it suffices to show that, for any $m$ and any two sequences $J$ and $J'$ in $\mathcal{I}_m$, we have 
\begin{equation}\label{eq:muVJ}
\mu_{J'}(B_{J})=
\left\{\begin{array}{ll}
1&\text{ if $J=J'$,}\\
0&\text{ otherwise. }
\end{array}
\right.
\end{equation}
(See \cite[Rem. 4.2]{MY}.)
Fixing a sequence $J=j_1...j_m$  in $\mathcal{I}_m$, set 
  $$ B_k = [[. . . [ [A_{j_1,j_2},A_{j_2,j_3}],A_{j_3,j_4}], . . . ] , A_{j_{k-1},j_k}]\in \Gamma_{k-1} P(l), $$ 
for all $k=2, . . . , m$. (In particular, $B_2=A_{j_1,j_2}$, while $B_{k}=B_J$.)  
Using the skein formula for Milnor invariants due to Polyak \cite{Polyak}, one can easily check that, for any $k=3, . . ., m$, we have 
 $$ \mu_{j_1. . . j_{k-1} j_k}(B_k)=\mu_{j_1. . . j_{k-1}}(B_{k-1}).$$
It follows that $\mu_J(B_J)=\mu_{j_1 j_2}(A_{j_1,j_2})=1$, as desired.  
The fact that $\mu_{J'}(B_J)=0$ for any $J'\neq J$ in $\mathcal{I}_m$ follows easily from similar arguments. 
\end{proof}

The following notation will be useful in the next sections. 
Let $SL_m^h(l)$ be the set of $l$-component  string links with vanishing  Milnor link-homotopy invariant of length $\leq m$, that is, 
 $L\in SL_m^h(l)$ if and only if $L$ is link-homotopic to $b^L_mb^L_{m+1}\cdots b^L_{l-1}$ as in Lemma \ref{lem:braidlh}.
Note that we have a descending filtration
 $$ SL(l)\supset SL^h_1(l)\supset SL^h_2(l)\supset . . . \supset SL^h_k(l)\supset . . . \supset SL^h_l(l).$$

\section{The universal $sl_2$ invariant}\label{3}

In the rest of this paper, we use the following $q$-integer notation.
\begin{align*}
&\{i\}_q = q^i-1,\quad  \{i\}_{q,n} = \{i\}_q\{i-1\}_q\cdots \{i-n+1\}_q,\quad  \{n\}_q! = \{n\}_{q,n},\\
&[i]_q = \{i\}_q/\{1\}_q,\quad  [n]_q! = [n]_q[n-1]_q\cdots [1]_q, \quad \begin{bmatrix} i \\ n \end{bmatrix} _q  = \{i\}_{q,n}/\{n\}_q!,
\end{align*}
for $i\in \mathbb{Z}, n\geq 0$.

\subsection{Quantized enveloping algebra $U_{\hbar}(sl_2)$}

We first recall  the definition of the quantized enveloping algebra $U_{\hbar}(sl_2)$, following the notation in \cite{H2,sakie2}.

We denote by  $U_{\hbar}=U_{\hbar}(sl_2)$ the $\hbar$-adically complete $\mathbb{Q}[[\hbar]]$-algebra,
topologically generated by  $H, E,$ and $F$, defined by the relations
\begin{align*}
HE-EH=2E, \quad HF-FH=-2F, \quad EF-FE=\frac{K-K^{-1}}{q^{1/2}-q^{-1/2}},
\end{align*}
where we set 
\begin{align*}
q=\exp {\hbar},\quad K=q^{H/2}=\exp\frac{{\hbar}H}{2}.
\end{align*}
\noindent
We equip $U_{\hbar}$  with a topological $\mathbb{Z}$-graded algebra structure with  $\deg F=-1$,  $\deg E=1$, and   $\deg H=0$.

There is a unique  complete ribbon Hopf algebra  structure  on  $U_{\hbar}$   such that
\begin{align*}
\Delta_{\hbar} (H)&=H\otimes 1+1\otimes H, \quad  \varepsilon_{\hbar}  (H)=0, \quad S_{\hbar} (H)=-H,
\\
\Delta_{\hbar}  (E)&=E\otimes 1+K\otimes E, \quad \varepsilon_{\hbar}  (E)=0, \quad  S_{\hbar} (E)=-K^{-1}E,
\\
\Delta_{\hbar}  (F)&=F\otimes K^{-1}+1\otimes F, \quad  \varepsilon_{\hbar}  (F)=0, \quad  S_{\hbar} (F)=-FK.
\end{align*}

The \emph{universal $R$-matrix} and its inverse are given by
\begin{align}\
R&=D\bigg(\sum_{n\geq 0}q^{\frac{1}{2}n(n-2)}\frac{(q-1)^n}{[n]_q!}F^n\otimes E^n\bigg),\label{rm1}
\\
R^{-1}&=\bigg(\sum_{n\geq 0}(-1)^nq^{-\frac{n}{2}}\frac{(q-1)^n}{[n]_q!}F^n\otimes E^n\bigg)D^{-1},\label{rm2}
\end{align} 
where $D=q^{\frac{1}{4}H\otimes H} =\exp \big(\frac{\hbar}{4}H\otimes H\big)\in U_{\hbar}^{\hat {\otimes }2}$.
For simplicity, we set  $R^{\pm 1}= \sum_{n\geq 0} \alpha^{\pm}_n \otimes \beta^{\pm}_n$ with
\begin{align*}
\alpha _n \otimes \beta _n(&=\alpha^+ _n \otimes \beta^+ _n)=D\Big(q^{\frac{1}{2}n(n-2)}\frac{(q-1)^n}{[n]_q!}F^n\otimes E^n\Big),
\\
\alpha _n^- \otimes \beta _n^-&=D^{-1}\Big((-1)^nq^{-\frac{n}{2}}\frac{(q-1)^n}{[n]_q!}F^nK^n\otimes K^{-n}E^n\Big).
\end{align*}
Note that the right-hand sides above are  sums of infinitely many tensors of the form  $x\otimes y$ with $x,y\in U_{\hbar}$, which we denote by  $\alpha ^{\pm}_n \otimes \beta^{\pm} _n$ formally.

\subsection{Universal $sl_2$ invariant for string links}\label{univinv}

In this section, we recall the definition of the universal $sl_2$ invariant of  string links.

For  an $n$-component string link $L=L_1\cup \cdots \cup L_n$, we define the universal $sl_2 $ invariant $J(L)\in U_{\hbar}^{\hat {\otimes }n}$  in three steps as follows. We follow the notation in \cite{sakie2}.

\textbf{Step 1. Choose a diagram.}
We choose a diagram $\tilde L$ of $L$ which is obtained by pasting, horizontally and vertically, 
copies of the fundamental tangles depicted in Figure \ref{fig:fundamental}.  
We call such a diagram \textit{decomposable}.
\begin{figure}[!h]
\centering
\includegraphics[width=9cm,clip]{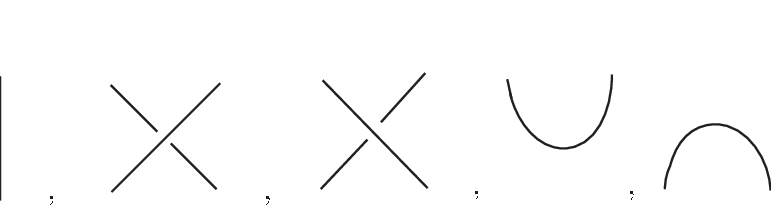}
\caption{Fundamental tangles, where the orientations of the strands are arbitrary.
 }\label{fig:fundamental}
\end{figure}

\textbf{Step 2.  Attach labels.}
We attach labels on the copies of the fundamental tangles in the diagram, 
following the rule described in Figure \ref{fig:cross}, where  $S_{\hbar}'$ should be replaced with $S_{\hbar}$
if the string is oriented upward, and with the identity otherwise.  
We do not attach any label to the other copies of fundamental tangles, i.e.,  to a straight strand and to a local maximum or minimum oriented from right to left.
See Figure \ref{fig:T_h} for an (elementary) example.

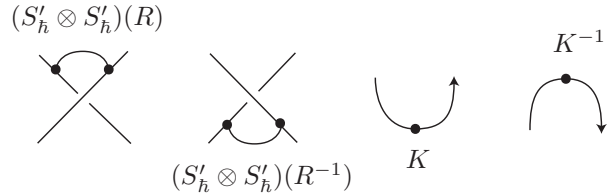
\begin{figure}[!h]
\centering
\begin{picture}(300,70)
\put(50,25){\includegraphics[width=7.5cm,clip]{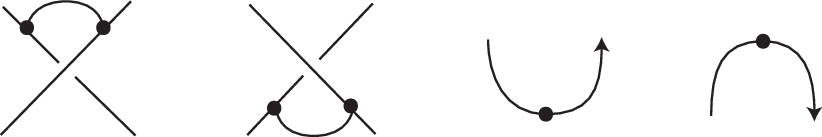}}
\put(40,70){$(S_{\hbar}'\otimes S_{\hbar}')(R)$}
\put(100,10){$(S_{\hbar}'\otimes S_{\hbar}')(R^{-1})$}
\put(188,16){$K$}
\put(243,60){$K^{-1}$}
\end{picture}
\caption{How to place labels  on the fundamental tangles.}\label{fig:cross}
\end{figure}

\textbf{Step 3.  Read the labels.}
We define the $i$th tensorand of $J({L})$ as the product 
of the labels on the $i$th component of $\tilde L$, where the labels are read off along $L_i$
reversing the orientation, and  written from left to  right.
Here, the labels on the crossings are read as in Figure   \ref{fig:cross2}.

\begin{figure}[!h]
\centering
\begin{picture}(300,120)
\put(50,20){\includegraphics[width=5.8cm,clip]{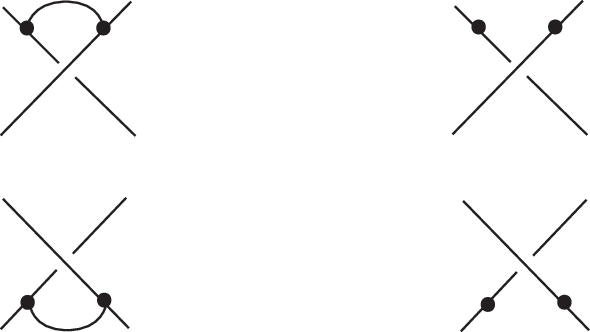}}
\put(42,120){$(S_{\hbar}'\otimes S_{\hbar}')(R)$}
\put(100,90){$=$}
\put(120,90){$\sum_{n\geq 0}$}
\put(150,95){$S_{\hbar}'(\alpha_n)$}
\put(208,95){$S_{\hbar}'(\beta_n)$}

\put(42,10){$(S_{\hbar}'\otimes S_{\hbar}')(R^{-1})$}
\put(100,40){$=$}
\put(120,30){$\sum_{m\geq 0}$}
\put(150,35){$S_{\hbar}'(\alpha^{-}_m)$}
\put(208,35){$S_{\hbar}'(\beta^{-}_m)$}
\end{picture}
\caption{How to read the labels on crossings.}\label{fig:cross2}
\end{figure}

As is well known \cite{O},  $J(L)$ does not depend on the choice of the diagram, and thus defines an isotopy invariant of string links.

\begin{figure}[!h]
\centering
\begin{picture}(300,100)
\put(70,20){\includegraphics[width=4.5cm,clip]{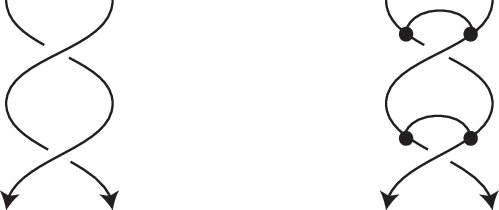}}
\put(35,40){$\tilde A \ =$}
\put(200,60){$R$}
\put(200,30){$R$}
\put(60,0){(a)}
\put(180,0){(b)}
\end{picture}
\caption{(a) A diagram $\tilde{A}$ of the string link $A$ (b) The label put on $\tilde{A}$. }\label{fig:T_h}
\end{figure}
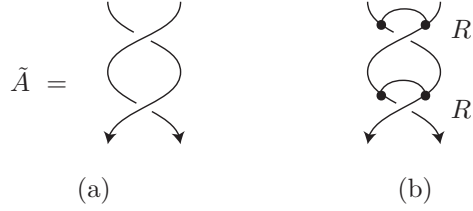

For example, for the string link $A$ shown in Figure \ref{fig:T_h}, we have
\begin{align}
\begin{split}\label{univc}
J({A})&=\sum_{m,n\geq 0} \beta_m\alpha_n\otimes \alpha_m\beta_n
\\
&=D\bigg(\sum_{m\geq 0}q^{\frac{1}{2}m(m-2)}\frac{(q-1)^m}{[m]_q!}E^m\otimes F^m\bigg)
D\bigg(\sum_{n\geq 0}q^{\frac{1}{2}n(n-2)}\frac{(q-1)^n}{[n]_q!}F^n\otimes E^n\bigg)
\\&=D^2\bigg(\sum_{m,n\geq 0}q^{\frac{1}{2}m(m-2)+{\frac{1}{2}n(n-2)+m^2}}\frac{(q-1)^{m+n}}{[m]_q![n]_q!} E^mK^mF^n\otimes F^mK^{-m}E^n\bigg),
\end{split}
\end{align} 
where the last identity follows from 
\begin{align*}
&D(1\otimes x)=(K^{|x|}\otimes x)D,
\end{align*}
for $x\in U_{\hbar}$ an homogeneous element of degree $|x|$.

Note that 
\begin{align*}
J({A})\equiv1+c\hbar \oh{2}, 
\end{align*}
where
 $c$ denotes the symmetric element
\begin{equation}\label{cs}
c=\frac{1}{2}H\otimes H+F\otimes E +E\otimes F.
\end{equation}

\subsection{Universal $sl_2$ invariant and linking number}\label{sec:Jlk}

We now recall how the linking number and framing can be simply derived from the ``coefficient'' of $\hbar$ in  the universal $sl_2$ invariant.
Before giving a precise statement (Proposition \ref{sc}), we need to introduce a few extra notation, which will be used throughout the paper.

For $1\leq i\leq n$,  and  for  $x\in U_{\hbar}$, we define
$x^{(l)}_i\in U_{\hbar}^{\hat \otimes l }$  by
$$x^{(l)}_i=1\otimes \cdots \otimes x
\otimes \cdots \otimes 1,
$$
where $x$ is at the $i$th position.

More generally, for $1\leq j_1,\ldots, j_m \leq l$  and $y=\sum y_1\otimes \cdots \otimes  y_m\in U_{\hbar}^{\hat \otimes m}$, 
we define  $y^{(l)}_{j_1\ldots j_m} \in U_{\hbar}^{\hat \otimes l }$ by
\begin{align*}
y^{(l)}_{j_1\ldots j_m}=\sum (y_1)^{(l)}_{j_1}\cdots (y_m)^{(l)}_{j_m}.
\end{align*}

For $x\in U_{\hbar}^{\hat \otimes l}$ such that $x\equiv1\oh{}$, set 
\begin{align*}
\coef(x)=\frac{x-1}{\hbar} \in U_{\hbar}^{\hat \otimes l} /\hbar U_{\hbar}^{\hat \otimes l},
\end{align*}
i.e., we have $x\equiv 1+\coef(x)\hbar \oh{2}$.

Note that $J(L)\equiv1 \oh{}$ for any string link $L$, by definition.

\begin{proposition}\label{sc}
For $L\in SL(l)$ with linking matrix $(m_{ij})_{1\leq i,j\leq l}$,
we have
\begin{align*}
\coef(J(L))&=\frac{1}{2} \sum_{1\leq i,j \leq l} m_{ij}c^{(l)}_{ij} 
\\
&=\sum_{1\leq i<j \leq l} m_{ij}c^{(l)}_{ij}+\frac{1}{2} \sum_{1\leq i \leq l} m_{ii}c^{(l)}_{ii}.  
\end{align*}
\end{proposition}
\begin{remark}
This result is similar to the well-known formula expressing the degree one part of the (framed) Kontsevich integral in terms of the linking matrix, 
which is generalized by Habegger and Masbaum \cite[Thm. 6.1]{HMa} with respect to Milnor invariants, where the case $n=1$  corresponds to the formula for the  linking matrix:
\begin{align} 
  Z^t(L) = 1 + \mu_{1}(L) + (\textrm{terms of degree $\ge 2$}) \in  \mathcal{A}^t(l), 
\end{align}
noting that Milnor numbers of length $2$ are the coefficients of the linking matrix (see Example \ref{ex:mu1}).  
Our main result in this paper generalizes Proposition \ref{sc} with respect to Milnor invariants, in a similar way to  \cite[Thm. 6.1]{HMa}. 
\end{remark}
In the rest of this section, we prove Proposition \ref{sc} in an elementary way.

\begin{proof}[Proof of Proposition \ref{sc}]
Let  $L\in SL(l)$, and choose a decomposable diagram $\tilde L=\tilde L_1\cup \cdots \cup \tilde L_l$ such that each crossings has both strands oriented downwards.
Denote by $C(\tilde L)$ the set of the crossings, and by $M(\tilde L)$  the set of  local maxima and minima oriented from left to right.
For   $a\in C(\tilde L) \cup M(\tilde L)$, let $J(a)\in U_{\hbar}^{\hat \otimes l}$ be the element obtained by reading only the labels on $a$, as indicated in Step 2
of the definition of $J(L)$.
Note that $J(a)\equiv1 \oh{}$ for each $a\in C(\tilde L) \cup M(\tilde L)$, and 
\begin{equation}\label{scl}
\coef(J(L))=\sum _{a\in C(\tilde L)\cup M(\tilde L)}\coef(J(a)).
\end{equation}

Now, for $1\leq i< j\leq l$, let $C_{i,j}(\tilde L)\subset C(\tilde L)$ be the subset of crossings between $\tilde L_i$ and $\tilde L_j$.
Set $R_{21}=R_{21}^{(2)}$. Since  we have
\begin{align*}
\coef(R^{\varepsilon})+\coef(R^{\varepsilon}_{21})=\varepsilon c, 
\end{align*}
for $\varepsilon=\pm 1$,  it follows that 
\begin{align}\label{cij}
\sum _{a\in C_{i, j}(\tilde L)}\coef(J(a)) =m_{ij}c^{(l)}_{ij}.
\end{align} 

Similarly, for $1\leq i\leq l$, let $C_{i}(\tilde L)\subset C(\tilde L)$ be the subset of self-crossings of $\tilde L_i$, and 
 $M_{i}(\tilde L)\subset M(\tilde L)$  the subset of local maxima and minima oriented from left to right in $\tilde L_i$.
Let us consider $J(a)$ for $a\in C_i(\tilde L)\cup M_i(\tilde L)$, for $l=i=1$ for simplicity.
Notice that each crossing in $C_{1}(\tilde L)$ is either left-connected or right-connected,  where a downward oriented crossing is called 
left (resp. right)-connected if its left (resp. right) outgoing strand is connected to the left (resp. right) ingoing strand in $\tilde L$. 
For a left (resp. right)-connected positive crossing $a\in C_1(\tilde L)$, we have $J(a)=R^{(1)}_{11}$ (resp. $J(a)=(R_{21})^{(1)}_{11}$),
and on a left (resp. right)-connected negative crossing $b\in C_1(\tilde L)$, we have $J(b)=(R_{21}^{-1})^{(1)}_{11}$ (resp. $J(b)=(R^{-1})^{(1)}_{11}$).
Recall that we put $K$ (resp. $K^{-1}$) on a local maximum (resp. minimum) oriented left to right.
For these labels we have
\begin{align*}
\coef(R^{(1)}_{11})&=\frac{c^{(1)}_{11} - H}{2},
\quad \coef((R^{-1})^{(1)}_{11})=\frac{-c_{11}^{(1)} + H}{2},
\\
 \coef((R_{21})^{(1)}_{11})&=\frac{c_{11}^{(1)} + H}{2},
\quad \coef((R_{21}^{-1})^{(1)}_{11})=\frac{-c_{11}^{(1)} - H}{2},
\\
\coef(K)&=\frac{ H}{2}, \quad \coef(K^{-1})=\frac{-H}{2}.
\end{align*} 
We consider the sum of these coefficients over all labels on $C_1(\tilde L)\cup M_1(\tilde L)$.
Actually, if $l$ (resp. $r$) denotes the number of left-connected (resp. right-connected) crossings in $C_{1}(\tilde L)$, and if 
$M$ (resp. $m$) denotes the number of local maximum (resp. minimum)  in $M_{1}(\tilde L)$, then it is not difficult to check that 
\begin{equation}\label{wutang}
   l - r - M + m= 0
\end{equation}
\noindent (By \cite[Thm. XII.2.2]{Kassel}, and since $l- r - M + m$ is clearly invariant under a crossing change, it suffices to prove that 
this quantity is invariant under each of the moves of  \cite[Fig. 2.2--2.9]{Kassel}:  
this is easily checked by a case-by-case study of all possible types of crossings involved in the moves).  
This implies that
\begin{align*}
\sum _{a\in C_{1}(\tilde L)}\coef(J(a))+\sum_{b\in M_1(\tilde L)} \coef(J(b))= \frac{1}{2}m_{11}c^{(1)}_{11}.
\end{align*} 
This,  together with Equations (\ref{cij})  and (\ref{scl}), implies the desired formula. 
\end{proof}

\section{Diagrammatic approach}\label{4}
\subsection{Jacobi diagrams} \label{sec:jacobi}

We mostly follow the notation in \cite{HMa}.

A \emph{Jacobi diagram} is a finite uni-trivalent graph, such that each trivalent vertex is equipped with a cyclic ordering of its three incident half-edges. 
In this paper we require that each connected component of a Jacobi diagram has at least one univalent vertex. 
The \emph{degree} of a Jacobi diagram is half its number of  vertices.

Let $X$ be a compact oriented $1$-manifold. A \emph{Jacobi diagram on $X$} is a Jacobi diagram whose univalent vertices are disjointly embedded in $X$. 
Let $\mathcal{A}(X)$ denote the $\mathbb{Q}$-vector space spanned by Jacobi diagrams on $X$, subject to the 
AS, IHX and STU relations depicted in Figure \ref{relations}.
Here as usual \cite{BN}, we use bold lines to depict the $1$-manifold $X$ and dashed ones to depict the Jacobi 
diagram, and the cyclic ordering at a vertex is given by the counter-clockwise orientation in the plane of the figure. 
\begin{figure}[!h]
\centering
\includegraphics{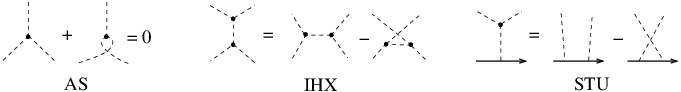}
\caption{The relations AS, IHX and STU. } \label{relations}
\end{figure}
We denote by $\mathcal{A}_k(X)$ the subspace spanned by Jacobi diagrams of degree $k$.
Abusing notation, we still denote by $\mathcal{A}(X)$ its completion with respect to the degree, i.e., $\mathcal{A}(X)=\prod_{k\geq 0} \mathcal{A}_k(X)$. 

In this paper we shall restrict our attention to the case $X=\coprod_{j=1}^l I_j$,
where each $I_j$ is a copy of the interval $I=[0,1]$.
For simplicity, set  $\mathcal{A}(l)=\mathcal{A}(\coprod_{j=1}^l I_j)$.
Note that $\mathcal{A}(l)$ has an algebra structure with multiplication defined by stacking. 

We denote by $\mathcal{B}(l)$ the completed $\mathbb{Q}$-vector space spanned by Jacobi diagrams whose univalent vertices are labelled by elements of the set $\{1,...,l\}$, subject to the  AS and IHX relations.
 Here completion is given by the degree as before.
Note that $\mathcal{B}(l)$ has an algebra structure with multiplication defined by  disjoint union.

There is a natural graded $\mathbb{Q}$-linear  isomorphism  \cite{BN}
  \begin{equation*} 
    \chi : \mathcal{B}(l)\to \mathcal{A}(l),
  \end{equation*}
 which maps a diagram to the average of all possible combinatorially 
distinct ways of attaching its $i$-colored vertices to the $i$th interval, for $i=1,\ldots , l$. 
Note that $\chi$ is not an algebra homomorphism.

In what follows, we focus only on the subspace $\mathcal{A}^t(l)$ of $\mathcal{A}(l)$, 
which is the graded quotient of $\mathcal{A}(l)$ by the space spanned  by Jacobi diagrams containing non-simply connected diagrams.   It follows that $\mathcal{B}^t(l)=\chi^{-1}(\mathcal{A}^t(l))$ is the commutative polynomial algebra on the subspace $\mathcal{C}^t(l)$ spanned by \emph{trees}, that is, by connected and simply connected Jacobi diagrams.

Let us also denote by $\mathcal{A}^h(l)$ the graded quotient of $\mathcal{A}^t(l)$ by the space spanned  by Jacobi diagrams containing a chord between the same component of $\coprod_{j=1}^l I_j$. 
Similarly, denote $\mathcal{B}^h(l):=\chi^{-1}(\mathcal{A}^h(l))$.  Then $\mathcal{B}^h(l)$ is the commutative polynomial algebra on the subspace  $\mathcal{C}^h(l)$ spanned by trees with distinct labels \cite{BN}.

As above, we denote by $\mathcal{C}^t_k(l)$ and $\mathcal{C}^h_k(l)$ the respective subspaces of $\mathcal{C}^t(l)$ and $\mathcal{C}^h(l)$ spanned by Jacobi diagrams of degree $k$.

For any sequence $I=(i_1,\ldots, i_{m})$ of integers in  $\{1,\ldots, l\}$, let $T^{(l)}_I$ be the tree Jacobi diagram of degree $(m-1)$ labeled by $I$ as shown in Figure \ref{fig:TI}.
\begin{figure}[!h]
\centering
  \input{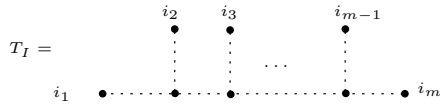}
  \caption{The tree Jacobi diagram $T^{(l)}_I$ for $I=(i_1,\ldots,i_{m})$. }\label{fig:TI}
\end{figure}
It is not difficult to see, by the AS and IHX relations, that  $\mathcal{C}_m^t(l)$ is spanned by diagrams $T^{(l)}_I$ indexed by sequences 
$I=(i_1,\ldots, i_{m})$ of integers in  $\{1,\ldots, l\}$,
while $\mathcal{C}_m^h(l)$ is spanned by those with distinct integers.

\subsection{Kontsevich integral and Milnor map}\label{sec:HM}

A \textit{tangle} is a proper embedding in $D^2\times [0,1]$ of a compact, oriented $1$-manifold $X$, whose
boundary points are on the two parallel lines $[-1,1]\times \{0\} \times \{0\}$ and $[-1,1]\times  \{0\} \times \{1\}$, where we use the parametrization $D^2=\{(x,y)\in \mathbb{R}^{2} | \sqrt{x^2+y^2} \leq 1 \}$. 
We further require that a tangle is equipped with a framing
and with a q-structure, i.e., a consistent collection of parentheses on each of the two sets of boundary points. 
In the case where $T$ is a string link, we assume that the q-structure is the same on both sets of boundary
points.\footnote{As we will see below, we only consider here the first non-vanishing term of $Z(T)-1$, which does not depend on this q-structure. }

The \emph{Kontsevich integral} $Z(T)$ of a  framed q-tangle $T$ lives in the space $\mathcal{A}(X)$ of Jacobi diagrams on $X$ \cite{Ko}. 
We shall not review the definition here, but refer the reader to \cite{BN,CD,ohtsuki} for surveys. 
In this paper, we use the combinatorial definition  as in \cite[Sec. 6.4]{ohtsuki}.\footnote{Note in particular that what is denoted here by $Z(T)$ is denoted by $\hat{Z}(T)$ in \cite{ohtsuki}. } Although this invariant depends on the choice of an associator, this choice will not be relevant in this paper. 

A fundamental property of the Kontsevich integral is its universality, over $\mathbb{Q}$, among finite type (or Vassiliev) invariants and among quantum invariants, in the sense that any such invariant can be recovered from the Kontsevich integral by post-composition with an appropriate map, called \emph{weight system}.

Bar-Natan \cite{BN2} and Lin \cite{Lin} proved that Milnor invariants for string links are finite type invariants, and thus can be recovered from the Kontsevich integral. 
This connection was made completely explicit by Habegger and Masbaum, who  showed that Milnor invariants determine and are determined 
by the so-called tree-part of the Kontsevich integral \cite{HMa}. 
In order to state this result, we first need the following diagrammatic formulation for the image of the Milnor map defined in Section \ref{sec:milnormap}.  

Denote by $H$ the abelianization $\F / \Gamma_2 \F$ of the free group $\F$, and denote by $L(H)=\bigoplus_k L_k(H)$ the free $\mathbb{Q}$-Lie algebra on $H$. 
Note that $L_k(H)$ is isomorphic to  $(\Gamma_k \F / \Gamma_{k+1} \F)\otimes \mathbb{Q}$, so that $\mu_{k}$ can be regarded as taking values in $H\otimes L_k(H)$. 
It turns out that the Milnor map $\mu_{k}$ actually takes values in the kernel $D_k(H)$ of the Lie bracket map 
$H\otimes L_k(H)\rightarrow L_{k+1}(H)$, and that $D_k(H)$ identifies 
with the space $\mathcal{C}^t_{k}(l)$, as we now explain. 
Let $T$ be a tree Jacobi diagram in $\mathcal{C}^t_{k}(l)$. 
To each univalent vertex $v_0$ of $T$, we associate an element $T_{v_0}$ of $L_{k}(H)$ as follows. 
For any univalent vertex $v\neq v_0$, label the incident edge of $T$ by $c_v=\alpha_j\in \F$, where $j$ is the label of $v$. 
Next, label all edges of $T$ by recursively assigning the label 
$[a,b]$ to any edge which meets an $a$-labelled and a $b$-labelled edge at a trivalent vertex (following the cyclic ordering). 
The last step of this process assigns a label to the edge incident to $v_0$: this final label is the desired element $T_{v_0}$ of $L_{k}(H)$. 
Using this, we can define a $\mathbb{Q}$-linear isomorphism 
 \begin{equation}\label{isomu}
   \mathcal{C}^t_{k}(l)\rightarrow D_k(H)
 \end{equation}
by sending a tree $T$ to the sum $\sum_v c_v\otimes T_v$, where the sum ranges over the set of all univalent vertices of $T$. 
\begin{example}\label{ex:isom}
A single chord with vertices labeled $i$ and $j$ is mapped to $\alpha_i\otimes \alpha_j + \alpha_j\otimes \alpha_i$, which is an element of $D_1(H)$ by antisymmetry. (Note in particular that for $i=j$, the corresponding diagram is mapped to $2.\alpha_i\otimes \alpha_i$.) \\
Similarly, an $Y$-shaped diagram with univalent vertices labeled $i$, $j$ and $k$ (following the cyclic ordering) 
is mapped to the sum $\alpha_i\otimes [\alpha_j,\alpha_k] + \alpha_j\otimes [\alpha_k,\alpha_i] + \alpha_k\otimes [\alpha_i,\alpha_j]$ : 
clearly, this is an element of $D_2(H)$ by the Jacobi identity. 
\end{example}
In the rest of this paper, we implicitly identify the image of the Milnor map $\mu_{k}$ with $\mathcal{C}^t_k(l)$ via the isomorphism (\ref{isomu}).  

Now, Habegger-Masbaum's result can be simply formulated as follows. 
Let $L\in SL_m(l)$ be an $l$-component string link with vanishing Milnor invariants of length up to $m$. 
The tree-part of the Kontsevich integral of $L$, which is defined as $Z^t=p^t\circ \chi^{-1}\circ Z$, 
where $p^t:\mathcal{B}(l)\rightarrow \mathcal{B}^t(l)$ is the natural projection, is then given by 
 \begin{align*}
Z^t(L) = 1 + \mu_{m}(L) + \textrm{terms of degree $\ge m+1$}, 
  \end{align*}
where $1$ denotes the empty Jacobi diagram.
In particular, the leading term of $Z^t-1$ does not depend on the choice of q-structure, and lives in the space $\mathcal{C}^t_m(l)$ of degree $m$ tree Jacobi diagrams.  

In \cite{HMa}, it is also proved that $Z^t$ is the universal finite type concordance invariant over $\mathbb{Q}$, which implies in particular that it determines Milnor invariants. 

Furthermore, Habegger and Masbaum showed that, for $L\in SL^h_m(l)$, we have 
\begin{align}\label{equ:HMh}
  Z^h(L) = 1 + \mu^h_{m}(L) + \textrm{terms of degree $\ge m+1$}, 
\end{align}
where $Z^h$ is the Kontsevich integral $Z$ composed with the projection $\mathcal{B}(l)\rightarrow \mathcal{B}^h(l)$ \cite{HMa}, and where $\mu^h_{m}$ is the link-homotopy reduction of the Milnor map $\mu_{m}$, which is defined as 
$\mu^h_{m} =  p^h\circ \mu_{m}$, with $p^h:\mathcal{C}^t(l)\rightarrow \mathcal{C}^h(l)$ the natural projection. 
(Note in particular that the leading term of $Z^h-1$ lives in $\mathcal{C}^h_m(l)$.) 

\subsection{Weight system associated to $sl_2$}\label{sec:weight}

Recall that the Lie algebra $sl_2$ is the 3-dimensional Lie algebra over $\mathbb{Q}$ generated by $h, e,$ and $f$ with Lie bracket
\begin{align*}
[h,e]=2e, \quad [h,f]=-2f, \quad [e,f]=h.
\end{align*}

Let $U=U(sl_2)$ denote the universal enveloping algebra of $sl_2$, and $S=S(sl_2)$ the symmetric algebra of $sl_2$.
There is a well-known commutative diagram  \cite{wheels} 
\begin{align*}
\begin{CD}
\mathcal{A}(l)  @>{W}>> U^{\otimes l} [[\hbar]]  \\
@AA{ \chi}A  @AA {\beta}A  \\
\mathcal{B}(l)  @>{W}>> S^{\otimes l}[[\hbar]],
\end{CD}
\end{align*}
where $\chi$ is the isomorphism defined in Section \ref{sec:jacobi},  $\beta$ is the $\mathbb{Q}$-linear isomorphism induced by the Poincar\'e-Birkhoff-Witt isomorphism 
$S \cong U$, sending a monomial $v_1\cdots v_m\in S$ to $\sum_{\sigma \in S(m)} \frac{1}{m!} v_{\sigma(1)}\cdots v_{\sigma(m)}\in U$, 
and where $W$ is the weight system associated to $sl_2$. 

In this paper we will make use of the map $W$ defined on the space $\mathcal{B}(l)$ of labeled Jacobi diagrams, and more precisely of its restriction to $\mathcal{B}^t(l)$, and thus recall its definition below. 
More precisely, we first define a map $w_m\co \mathcal{C}^t_m(l) \to S^{\otimes l}$, and then 
define $W\co \mathcal{B}^t(l) \to S^{\otimes l}[[\hbar]]$ as the $\mathbb{Q}$-algebra homomorphism such that  $W(D)=w_m(D)\hbar^m$ for $D\in  \mathcal{C}^t_m(l)$.

For  $m=1$, we simply define $w_1$ by
\begin{align}
 w_1(D_{ij}) = c_{ij}^{(l)}\in  S^{ \otimes l},
\end{align}
where $D_{ij}$ is a single chord with vertices labeled by $i$ and $j$ (possibly $i=j$), 
and where\footnote{Abusing notation, we denote by the same letter $c$ the element of $U_{\hbar}^{\otimes 2}$ defined in (\ref{cs}) and the corresponding element in $sl_2^{\otimes 2}$.}  
\begin{align}\label{csc}
c=\frac{1}{2}h\otimes h+f\otimes e +e\otimes f\in  sl_2^{\otimes 2}.
\end{align}

Now let $m\geq 2$, and let $D\in \mathcal{C}^t_m(l)$. 
Set 
\begin{align}\label{csb}
\begin{split}
b &= \sum_{\sigma\in \mathfrak{S}(3)} (-1)^{|\sigma |} \sigma (h\otimes e\otimes f)
\\
&=h\otimes e \otimes f + e\otimes f\otimes h + f\otimes h\otimes e 
 -h\otimes f\otimes e -f\otimes e\otimes h -e\otimes h \otimes f 
\in sl_2^{\otimes 3},
\end{split}
\end{align}
where $\sigma$ acts by permutation of the tensorands.
Consider a copy of $b$ for each trivalent vertex of $D$, 
where each tensorand of $b$ is associated to one of the half-edges incident to the trivalent vertex, 
following the cyclic ordering.   
Each internal edge (i.e.  each edge between two trivalent vertices)
comprises a pair of half-edges, and we contract the two corresponding copies of $sl_2$ using the symmetric bilinear form 
 $$ \langle - , - \rangle : sl_2\otimes sl_2 \rightarrow \mathbb{Q} $$
defined by $\langle a , b \rangle = Tr(ab)$, that is given by 
$$ \langle h , h \rangle = 2\quad , \quad \langle e , f \rangle =1\quad , \quad \langle h , e \rangle = \langle h , f \rangle = \langle e , e \rangle = \langle f , f \rangle = 0. $$
Fix an arbitrary total order on the set of univalent vertices of $D$; we get in this way an element $\sum x_1\otimes . . . \otimes x_{m+1}$ of $sl_2^{\otimes m+1}$, the $i$th tensorand corresponding to the $i$th univalent vertex of $D$. 
We then define $w_m(D) \in S^{\otimes l}$ by 
\begin{align}\label{formula}
w_m(D)= \sum y_1\otimes \cdots \otimes  y_l,
\end{align}
where $y_j$ is the product of all $x_{i}\in sl_2$ such that the $i$th vertex is labelled by $j$.

It is known that $w_m$ is well-defined, i.e. is invariant under AS and IHX relations; see e.g. \cite[Sec. 6.2]{CDM}. 

\subsection{Computing $w_m$ on trees}\label{sec:compute}

There is another formulation of $w_m$ for tree Jacobi diagrams, which we will use later.
Recall that $ \mathcal{C}^t_m(l)$ is spanned by the trees $T^{(l)}_I$ indexed by sequences $I$ of (possibly repeating) integers in  $\{1,\ldots, l\}$, introduced in Section \ref{sec:jacobi}. 
For convenience, we only give this alternative definition of $w_m$ on the trees $T^{(l)}_I$. 

Recall the element $c\in sl_2^{\otimes 2}$ and $b\in sl_2^{\otimes 3}$ defined in (\ref{csc}) and (\ref{csb}),
respectively.
Let 
$$s\co  sl_2  \to sl_2 ^{\otimes 2}$$ 
be the $\mathbb{Q}$-linear map defined by 
\begin{align*}
s(a)=(\mathrm{ad}\otimes 1)(a\otimes c)= \frac{1}{2} [a,h]\otimes h+[a,f]\otimes e+[a,e]\otimes f
\end{align*}
for $a\in sl_2$,
where $\mathrm{ad}(x\otimes y)=[x,y]$ for $x, y\in sl_2$.
On the basis elements, we have
\begin{align*}
s(e)&= h\otimes e - e\otimes h, 
\\
s(h)&=2(e\otimes f - f\otimes e),
\\
s(f) &  = f\otimes h - h\otimes f.
\end{align*}

Set $\varsigma_2=c\in  sl_2^{\otimes 2}$.  
For $m\geq 3$, set
\begin{align}\label{varsigma}
\varsigma_m=(1^{\otimes m-2}\otimes s)(1^{\otimes m-3} \otimes s)\cdots (1\otimes  s)(c) \in sl_2^{\otimes m}.
\end{align}
\noindent
For example, one can easily check that $\varsigma_3=b$.

\begin{proposition}\label{lem:Xm}
For $m\geq 1$, we have
\begin{align}\label{sch}
 w_m(T^{(m+1)}_{(1,\ldots, m+1)})= \varsigma_{m+1}.
 \end{align}\end{proposition}
\begin{proof}
This is easily shown by induction on $m\geq 1$.
For $m=1$, we have $w_1 (T^{(2)}_{(1,2)})=c=\varsigma_{2}$.
Now let $m\geq 2$, and let  $X_h, X_e,X_f\in sl_2 ^{\otimes m-2}$ such that
 \begin{align*}
\varsigma_{m} &=  X_h\otimes h +X_e\otimes e+X_f\otimes f
\\
&=w_{m-1}(T^{(m)}_{(1,\ldots, m)}).
 \end{align*}

Then we have
\begin{align*}
\varsigma_{m+1}&=(1^{\otimes m-1}\otimes s)(\varsigma_{m-1} )
\\
&= (1^{\otimes m-1}\otimes s)(X_h\otimes h +X_e\otimes e+X_f\otimes f)
\\
&=X_h\otimes 2 (e\otimes f-f\otimes e ) +X_e\otimes (h\otimes e -e\otimes f)+X_f\otimes (f\otimes h-h\otimes f)
\\
&=X_h\otimes  (\sum \langle h, b_1\rangle b_2\otimes b_3) +X_e\otimes (\sum \langle e, b'_1\rangle b'_2\otimes b'_3)+X_f\otimes  (\sum \langle f, b''_1\rangle b''_2\otimes b''_3)
\\
&=w_m(T^{(m+1)}_{(1,\ldots, m+1)}),
\end{align*}
where $\sum b_1\otimes b_2\otimes b_3=\sum b'_1\otimes b'_2\otimes b'_3=\sum b''_1\otimes b''_2\otimes b''_3=b$.

Hence we have the assertion.
\end{proof}

For an arbitrary sequence $I=(i_1,\ldots, i_{m+1})$ of indices in $\{1,\ldots, l\}$, set
\begin{align*}
\varsigma^{(l)}_{I}=(\varsigma_{m+1})^{(l)}_{i_1,\ldots, i_{m+1}},
\end{align*}
that is,  if we write formally $\varsigma_{m+1}=\sum x_1\otimes \cdots \otimes x_{m+1}$, the $j$th tensorand of $\varsigma^{(l)}_{I}$ is the product of 
 of all $x_{p}\in U$ such that $i_p=j$.
Then  by Proposition \ref{lem:Xm} and the definition of $w_m$, it  immediately follows that
\begin{align}\label{formula2}
w_m(T^{(l)}_I)= \varsigma^{(l)}_{I}\in S^{\otimes l}. 
\end{align}

\section{Milnor map and the universal $sl_2$ invariant}\label{5}

In this section we give the  main results of this paper, which relate Milnor invariants to the universal $sl_2$ invariant
via the $sl_2$ weight system $W$. 

\subsection{The quantized  enveloping algebra $U_{\hbar}$ and formal power series $S[[\hbar]]$ over the symmetric algebra}\label{5.1}

The symmetric algebra $S$ of $sl_2$ has a graded structure $S=\oplus_{m\geq 0} S_m$, where $S_m$ is the $\mathbb{Q}$-subspace spanned by elements of homogeneous degree $m$. 
Likewise, its $l$-fold tensor product $S^{\otimes l}=\oplus_{m\geq 0} (S^{\otimes l})_m$ is graded, with  
\begin{align*}
(S^{\otimes l})_m=\bigoplus_{\substack{m_1+\cdots+m_l=m \\ m_1,\ldots,m_l\geq 0.}} S_{m_1}\otimes \cdots \otimes S_{m_l}.
\end{align*}
Consider the $\mathbb{Q}$-subspace $\langle sl_2\rangle ^{(l)}_{m}$ of $(S^{\otimes l})_m$ defined by
\begin{align*}
\langle sl_2\rangle ^{(l)}_{m}=\bigoplus_{\substack{m_1+\cdots+m_l=m \\ 0\leq m_1,\ldots,m_l\leq 1}} S_{m_1}\otimes \cdots \otimes S_{m_l}.
\end{align*}
Roughly speaking, $\langle sl_2\rangle ^{(l)}_{m}$ is spanned by tensors in $S^{\otimes l}$ with exactly $m$ nontrivial tensorands, each of which being of degree one. 
For example, the tensor $\varsigma_m$ defined in (\ref{varsigma}) is an element of $\langle sl_2\rangle ^{(l)}_{m}$. 

By the definition of $W$ given in Sections \ref{sec:weight} and \ref{sec:compute}, we immediately have the following.
\begin{lemma}\label{wcc}
For $m\geq 1,$  we have 
$$W(\mathcal{C}^t_m(l))\subset (S^{\otimes l})_{m+1}\hbar^m \quad \textrm{and} \quad 
W(\mathcal{C}^h_m(l))\subset \langle sl_2\rangle ^{(l)}_{m+1}\hbar^m. $$
\end{lemma}

In what follows, we will identify $U_{\hbar}^{\hat \otimes l}$ and $U^{\otimes l}[[\hbar]]$ as $\mathbb{Q}[[\hbar]]$-modules via the isomorphism
\begin{align*}
\rho\co U_{\hbar}^{\hat \otimes l}\to  U^{\otimes l}[[\hbar]]
\end{align*}
defined by
\begin{align*}
&\rho \left(\sum_{s_{1},\ldots, s_{l},t_{1},\ldots, t_{l},u_{1},\ldots, u_{l}\geq 0} a_{s_{1},\ldots, s_{l},t_{1},\ldots, t_{l},u_{1},\ldots, u_{l}} F^{s_{1}}H^{t_{1}}E^{u_{1}}\otimes \cdots \otimes F^{s_{l}}H^{t_{l}}E^{u_{l}}   \right)
\\
&= \sum_{s_{1},\ldots, s_{l},t_{1},\ldots, t_{l},u_{1},\ldots, u_{l}\geq 0} a_{s_{1},\ldots, s_{l},t_{1},\ldots, t_{l},u_{1},\ldots, u_{l}}  f^{s_{1}}h^{t_{1}}e^{u_{1}}\otimes \cdots \otimes f^{s_{l}}h^{t_{l}}e^{u_{l}},
\end{align*}
for $ a_{s_{1},\ldots, s_{l},t_{1},\ldots, t_{l},u_{1},\ldots, u_{l}} \in \mathbb{Q}[[\hbar]].$
We also identify  $U_{\hbar}^{\hat \otimes l}$ and $S^{\otimes l}[[\hbar]]$ as $\mathbb{Q}$-modules similarly.

\subsection{Main results}\label{subseq}
We can now give the precise statements of our main results. 

Set
\begin{align*}
J^t&:=\pi^t\circ J\co SL(l) \to  \prod_{m\geq 1}(S^{\otimes l})_{m+1} \hbar^m,
\end{align*}
where 
\begin{align*}
\pi^t\co U_{\hbar}^{\hat \otimes l} &\to   \prod_{m\geq 1}(S^{\otimes l})_{m+1} \hbar^m
\end{align*}
denotes the projection as $\mathbb{Q}$-modules, that is,
\begin{align*}
\pi^t&\left( \sum_{m\geq 0}\sum_{i\geq 0}\left(\sum_{\sum_{i=1}^l s_i+t_i+u_i=m}
b^{(i)}_{s_{1},\ldots, s_{l},t_{1},\ldots, t_{l},u_{1},\ldots, u_{l}} F^{s_{1}}H^{t_{1}}E^{u_{1}}\otimes \cdots \otimes F^{s_{l}}H^{t_{l}}E^{u_{l}}\right)\hbar^{i}\right)  
\\
&=   \sum_{m\geq 1}\left(\sum_{\sum_{i=1}^l s_i+t_i+u_i=m+1}
 b^{(m)}_{s_{1},\ldots, s_{l},t_{1},\ldots, t_{l},u_{1},\ldots, u_{l}}  f^{s_{1}}h^{t_{1}}e^{u_{1}}\otimes \cdots \otimes f^{s_{l}}h^{t_{l}}e^{u_{l}}\right)\hbar^{m},
\end{align*}
for $b^{(i)}_{s_{1},\ldots, s_{l},t_{1},\ldots, t_{l},u_{1},\ldots, u_{l}}\in \mathbb{Q}$.\\

The first main result in this paper is as follows.
\begin{theorem}\label{sth2}
Let $m\geq 1$.  If $L\in SL_m(l)$, then we have
$$
J^t(L)\equiv (W\circ \mu_{m}) (L) \oh{m+1}.  
$$
\end{theorem}

\begin{example}\label{ex:mu1}
Let $L$ be a string link with nonzero linking matrix $(m_{ij})_{1\le i,j \le l}$. 
Then the $i$th longitude $l^1_i$ in $\F/\Gamma_{2}\F$, defined in Section \ref{sec:milnormap}, reads 
$l^1_i = \sum_{j=1}^l  m_{ij} \alpha_j$ ($1\le i\le l$), 
so that the degree $1$ Milnor map of $L$ is given by 
\begin{align*}
 \mu_1(L) & = \sum_{i=1}^l \alpha_i\otimes l^1_i\\
  & = \sum_{1\le i,j \le l} m_{ij} \alpha_i\otimes \alpha_j\\
  & = \sum_{1\le i<j \le l} m_{ij} (\alpha_i\otimes \alpha_j + \alpha_j\otimes \alpha_i)+ \sum_{1\leq i \leq l} m_{ii} (\alpha_i\otimes  \alpha_i)\\
  & = \sum_{1\le i<j \le l} m_{ij}. D_{\alpha_i,\alpha_j} +  \sum_{1\leq i \leq l} m_{ii}. \frac{1}{2} D_{\alpha_i,\alpha_i},  
\end{align*}
where $D_{\alpha_i,\alpha_j}$ denotes a single chord with vertices labelled $\alpha_i$ and $\alpha_j$, and where the last equality uses isomorphism (\ref{isomu}), see Example \ref{ex:isom}. 
Applying the $sl_2$ weight system $W$ then yields 
 $$ (W\circ \mu_1)(L) = \sum_{1\leq i<j \leq l} m_{ij}c^{(l)}_{ij}+\frac{1}{2} \sum_{1\leq i \leq l} m_{ii}c^{(l)}_{ii}, $$
as predicted by Proposition \ref{sc}.  
\end{example} 

Since Milnor maps are concordance invariants, we obtain the following topological property for $J^t$ as an immediate consequence of Theorem \ref{sth2}. 
\begin{corollary}\label{cor:main}
Let $L, L' \in SL_m(l)$  be two concordant string links.
Then we have
 \begin{align*}
  J^t(L') \equiv J^t(L) \oh{m+1}. 
  \end{align*}
In particular, if $L$ is concordant to the trivial string link, then $J^t(L)$ is trivial.  
\end{corollary}
\noindent 

Theorem \ref{sth2} is proved in Section \ref{7}. 
The proof relies on the fact that Milnor concordance invariants are related to Milnor link-homotopy  invariants via some cabling operation, so that Theorem \ref{sth2h} below is actually used as a tool for proving Theorem \ref{sth2}. 

In order to state the second main result in this paper, set 
\begin{align*}
J^h&:=\pi^h\circ J\co SL(l) \to \bigoplus_{m=1}^{l-1}\langle sl_2 \rangle^{(l)}_{m+1}{\hbar}^m,  
\end{align*}
where 
$\pi^h\co U_{\hbar}^{\hat \otimes l} \to \bigoplus_{m=1}^{l-1}\langle sl_2 \rangle _{m+1}^{(l)}\hbar^m$ 
is the projection as $\mathbb{Q}$-modules. 

We have the following. 
\begin{theorem}\label{sth2h}
Let $m\geq 1$.  
If $L\in SL^h_m(l)$,  then we have 
$$
J^h(L)\equiv  (W\circ \mu^h_{m}) (L) \oh{m+1}, 
$$
where $\mu^h_{m}$ is the link-homotopy reduction of the Milnor map $\mu_{m}$, defined in Section \ref{sec:HM}.  
\end{theorem}

Theorem \ref{sth2h} is equivalent to the following  theorem, formulated in terms of Milnor numbers and the tensors $\varsigma^{(l)}_I$ defined in Section \ref{sec:compute}.  
\begin{theorem}\label{sth1}
For $m\geq 1$, if  $L\in SL_m^h(l)$, then
we have 
$$
J^h(L)\equiv \left( \sum_{I\in \mathcal{I}_{m+1}} \mu_I(L)\varsigma^{(l)}_I\right)\hbar^m \oh{m+1},    
$$
where the sum runs over the set $\mathcal{I}_{m+1}$ defined in Section 2.
\end{theorem}

\begin{proof}[Proof of equivalency of Theorem \ref{sth2h} and \ref{sth1}]
We need to prove that 
\begin{align} \label{equivalence}
(w_m \circ \mu^h_{m})(L)= \sum_{I\in \mathcal{I}_{m+1}} \mu_I(L) \varsigma^{(l)}_I.
\end{align}
Recall from Lemma \ref{lem:braidlh} that if  $L\in SL_m^h(l)$, then $L$ is link-homotopic to $b_m^Lb_{m+1}^L\cdots b_{l-1}^L$, where the pure braids $b^L_i$ are defined in (\ref{eq:braid}). 
Actually, it follows directly from Equation (\ref{eq:braid}) that if $L\in SL_m^h(l)$, then we have $\mu^h_{m} (L)=\mu^h_{m} (b_m^L)$.
We thus have that 
\begin{align*}
\mu^h_{m} (L)
&=\mu^h_{m} (b_m^L)
\\
&=\mu^h_{m} \left(\prod_{I\in \mathcal{I}_{m+1}} (B_{I})^{\mu_{I}(b^L_m)}\right)
\\
&=\mu^h_{m} \left(\prod_{I\in \mathcal{I}_{m+1}} (B_{I})^{\mu_{I}(L)}\right)
\\
&=\sum_{I\in \mathcal{I}_{m+1}} \mu_{I}(L).\mu^h_{m} (B_{I}),
\end{align*}
where the last equality uses the additivity of the first non-vanishing Milnor string link invariants. 
The result then follows from (\ref{formula2}) and the fact that $\mu^h_m(B_I)=T_I$ for $I\in \mathcal{I}_{m+1}$, 
which can be easily checked either by a direct computation or using (\ref{equ:HMh}).
\end{proof}
We prove Theorems \ref{sth1} in the next section.

\section{Proof of Theorem \ref{sth1} : the link-homotopy case}\label{6}

We reduce Theorem \ref{sth1} to the following two propositions.
The first one shows that the invariant $J^h$ is well-behaved with respect to link-homotopy.
\begin{proposition}\label{s241}
Let $L, L' \in SL^h_m(l)$  be two link-homotopic  string links.
Then we have
 \begin{align*}
  J^h(L') \equiv J^h(L) \oh{m+1}.  
  \end{align*}
In particular, if $L$ is link-homotopic to the trivial string link, then $J^h(L)$ is trivial.  
\end{proposition}

For the second proposition, 
recall from Section \ref{sec:pure} that for each sequence $J\in \mathcal{I}_{m+1}$ we defined a pure braid $B^{(l)}_J$ which lies in the $m$th term of the lower central series of $P(l)$. 
\begin{proposition}\label{sp}
For any $J\in \mathcal{I}_{m+1}$, we have 
\begin{align*}
J(B^{(l)}_{J})& \equiv  1+\varsigma^{(l)}_{J} {\hbar}^m \oh{m+1},
\end{align*}
where $\varsigma^{(l)}_{J} \in $ was defined in Section \ref{sec:compute}. 
\end{proposition}

\begin{proof}[Proof of Theorem \ref{sth1} assuming Propositions \ref{s241} and \ref{sp}]
We first note that, as an immediate consequence of Proposition \ref{sp},  for any $J\in \mathcal{I}_{m+1}$ we have 
\begin{align}\label{eq:sp}
J^h(B^{(l)}_{J})& \equiv  \varsigma^{(l)}_{J} {\hbar}^m \oh{m+1}.  
\end{align}
Now, let $L\in SL_m^h(l)$, for some $m\geq 1$. 
We have 
\begin{align*}
J^h(L)& \equiv  J^h(b^L_m\cdot b^L_{m+1}\cdot . . . \cdot b^L_{l-1}) 
\\
&\equiv  J^h(b^L_m) 
\\
&\equiv J^h \left( \prod_{I\in \mathcal{I}_{m+1}}B_{I}^{\mu_{I}(L)}  \right)
\\
& \equiv  \sum_{I\in \mathcal{I}_{m+1}}\mu_{I}(L)\varsigma^{(l)}_{I} {\hbar}^m \oh{m+1}, 
\end{align*}
where the first equality uses Lemma \ref{lem:braidlh} and Proposition \ref{s241}, while the last three equalities follow from the definition of 
the pure braids $b^L_i$ and from (\ref{eq:sp}). Thus we have the assertion. 
\end{proof}

The proof of Proposition  \ref{s241}, which makes use of the theory of claspers, is postponed to Section \ref{8}. 
Proposition \ref{sp} is proved by a direct computation, as shown below. 
\begin{proof}[Proof of Proposition \ref{sp}]
Set $J=j_1j_2\ldots j_{m+1}$
The result is shown by induction on $m$. 
For $m=1$, then $B^{(l)}_{j_1 j_2}$ is the pure braid $A^{(l)}_{j_1,j_2}$, so by Proposition \ref{sc}  we have
\begin{align*}
J(B^{(l)}_{j_1 j_2}) &\equiv 1 + c^{(l)}_{j_1, j_2}\hbar 
\\
&\equiv 1 + \varsigma^{(l)}_{j_1 j_2} \hbar \oh{2}, 
\end{align*}
as desired. 

For $m>1$, by the induction hypothesis we have
\begin{align*}
J(B^{(l)}_{J})&=J([B^{(l)}_{j_1\ldots j_m},A^{(l)}_{j_m,j_{m+1}}])
\\
& = [J(B^{(l)}_{j_1\ldots j_m}), J(A^{(l)}_{j_m,j_{m+1}})]
\\
&\in [1+\varsigma_{j_1\ldots j_m}^{(l)}{\hbar}^{m-1} + \hbar^{m}U_{\hbar}^{\hat\otimes l}, 1+c_{j_m, j_{m+1}}^{(l)}\hbar +\hbar^{2}U_{\hbar}^{\hat\otimes l}] 
\\
&\subset 1+[\varsigma_{j_1\ldots j_m}^{(l)},c_{j_m, j_{m+1}}^{(l)}]{\hbar}^{m} + \hbar^{m+1}U_{\hbar}^{\hat \otimes l},
\end{align*}
and on the other hand we have
\begin{align*}
[\varsigma_{j_1\ldots j_m}^{(l)},c_{j_m ,j_{m+1}}^{(l)}]&=\left([\varsigma_m\otimes 1,c^{(m+1)}_{m, {m+1}}]\right)^{(l)}_{j_1\ldots  j_{m+1}}
\\
&=\left((1^{\otimes m-2}\otimes \mathrm{ad} \otimes 1)(\varsigma_{m}\otimes c)\right)^{(l)}_{j_1\ldots j_{m+1}}
\\
&=\left((1^{\otimes m-2}\otimes s)(\varsigma_{m})\right)^{(l)}_{j_1\ldots  j_{m+1}}
\\
&=(\varsigma_{m+1})^{(l)}_{j_1\ldots  j_{m+1}}
\\
&=\varsigma_{J}^{(l)}.
\end{align*}
This completes the proof.
\end{proof}

\section{Proof of Theorem \ref{sth2} : the general case} \label{7}

In this section, we show how to deduce Theorem \ref{sth2} from Theorem \ref{sth2h}. 

First, let us set some notation for the various projection maps that will be used throughout this section. 
For $i,j\ge 1$, let 
  $$ \pi^t_{i,j} \co S^{ \otimes l}[[\hbar]]\to (S^{ \otimes l})_i \hbar^j, $$
where $(S^{ \otimes l})_i$ was defined in Section \ref{5.1}, and set also 
$ \pi^t_{i}:=\pi^t_{i+1,i}$, so that $\pi^t=\prod_{i\ge 1} \pi^t_{i}$. 
Likewise, let 
  $$ \pi^h_{i,j} \co S^{ \otimes l}[[\hbar]]\to \langle sl_2 \rangle^{(l)}_{i} \hbar^j, $$
and set $ \pi^h_{i}:=\pi^h_{i+1,i}$.  

Next, recall that there are two (completed) coalgebra structures on  $S[[\hbar]]$ as follows.  
The first one is defined by  $\bar \Delta (x)=x\otimes 1+1\otimes x$  and $\varepsilon (x)=0$ for $x\in sl_2$ as algebra morphisms. 
The second one is induced  by the coalgebra structure of  $U_{\hbar}$ defined in Section \ref{3}, via  the $\mathbb{Q}$-module isomorphism $\rho \co  U_{\hbar} \to S[[\hbar]]$ seen in Section \ref{5.1}.
For $\Delta=\Delta _{\hbar}, \bar \Delta$  and  $p\geq 0$,  define  
$$\Delta^{[p]} \co S[[\hbar]]\to S^{\otimes p}[[\hbar]]$$ 
by $\Delta^{[0]}=\varepsilon$, $\Delta^{[1]}=\id$, $\Delta^{[2]}=\Delta,$ and $\Delta^{[p]}= (\Delta \otimes 1^{\otimes p-2})\circ \Delta^{[p]}$ for $p\geq 3$.
Abusing notation, for $l\geq 0,$ we  write $\Delta^{(p)}:=(\Delta^{[p]})^{\otimes l}\co S^{ \otimes l}[[\hbar]]\to S^{\otimes pl}[[\hbar]]$.

Since we have $\Delta _{\hbar}(y)\equiv\bar \Delta(y) \oh{}$ for any $y\in S$, 
the restriction of $\pi_{i,j}^h\circ \bar \Delta^{(p)}$ to $(S^{\otimes l})_{i}\hbar^j$ is equal to that of
$\pi_{i,j}^h\circ \Delta_{\hbar}^{(p)}$ for any $1\le i,j < p$, that is, we have  
\begin{align}\label{eq:delta}
\pi_{i,j}^h\circ \bar \Delta^{(p)}|_{(S^{\otimes l})_{i}\hbar^j} =\pi_{i,j}^h\circ \Delta_{\hbar}^{(p)}|_{(S^{\otimes l})_{i}\hbar^j}.  
\end{align}
Actually, the injectivity of these maps is one of the key points in this section.  
\begin{lemma}\label{sl3}
For any $1\le i,j < p$, the restriction to $(S^{\otimes l})_{i}\hbar^j$ of the $\mathbb{Q}$-linear map  
$\pi_{i,j}^h \circ \bar \Delta^{(p)}$ is injective. 
\end{lemma}
\begin{proof}
This simply follows from the fact that the map $\nabla^{(p)}\circ \pi_{i,j}^h \circ \bar\Delta^{(p)}$ is a scalar map on $(S^{\otimes l})_{i}\hbar^j$, where  $\nabla^{(p)}\co U_{\hbar}^{\hat \otimes pl} \to U_{\hbar}^{\hat \otimes l}$  is the tensor power of $p$-fold multiplications, i.e. the map sending $x_1\otimes \cdots \otimes  x_{pl}\in U_{\hbar}^{\hat \otimes pl} $ to $x_1\cdots x_p\otimes \cdots \otimes  x_{p(l-1)+1}\cdots x_{pl}\in U_{\hbar}^{\hat \otimes l}$. 
\end{proof}

Now, for $p\geq 1$, let  $D^{(p)}: SL(l)\rightarrow SL(pl)$ be the cabling map,  which sends a string link $L\in SL(l)$ to the string link $D^{(p)}(L)\in SL(pl)$ obtained by replacing each component with $p$ parallel copies. 
Recall from \cite{HMa} that, for $m\ge 1$ and $p>m$, we have that $L\in SL_m(l)$ if and only if $D^{(p)}(L)\in SL^h_m(pl)$.  We have the following.
\begin{lemma}\label{scom}
For $1\le m < p$,  the following diagram commutes 
\begin{align*}
\xymatrixcolsep{5pc}
\xymatrix{
SL_m(l)\ar[d]_{D^{(p)}} \ar@{->}[r]^{ W\circ \mu_{m}} & (S^{\otimes l})_{m+1} \hbar^m\ar[d]^-{\pi^h_m \circ \bar \Delta^{(p)}} \\
SL^h_m(pl) \ar@{->}[r]_{W \circ \mu^h_{m}} & \langle sl_2 \rangle^{(pl)}_{m+1}\hbar^m. 
}
\end{align*}
\end{lemma}
\begin{proof}
Denote by $D ^{(p)}: \mathcal{C}^t_m(l)\rightarrow \mathcal{C}^t_m(pl)$ the map defined by sending a tree Jacobi diagram $\xi \in \mathcal{C}^t_m(l)$ to the  sum of all diagrams obtained from $\xi$ by replacing each label $i\in \{1,\ldots, l\}$ by one of ${(i-1)p+1, (i-1)p+2, \ldots, ip}$.
Then the lemma follows from the following two commutative diagrams
$$
\xymatrix{
SL_m(l)\ar[d]_{D^{(p)}} \ar@{->}[r]^{\mu_{m}} & \mathcal{C}^t_m(l)\ar[d]^{p^h\circ D^{(p)}} \\
SL^h_m(pl) \ar@{->}[r]_{\mu^h_{m}} & \mathcal{C}^h_m(pl),
}
 \textrm{ and }
\xymatrix{
\mathcal{C}^t_m(l)\ar[d]_{p^h\circ D^{(p)}} \ar@{->}[r]^{W\ \ \ } & (S^{\otimes l})_{m+1}\hbar^m \ar[d]^-{\pi^h_m\circ \bar \Delta^{(p)}} \\
\mathcal{C}^h_m(pl)\ar@{->}[r]_{W\ \ \ } & \langle sl_2 \rangle^{(pl)}_{m+1}\hbar^m.
}
$$
The fact that the left-hand side diagram commutes is due to Habegger and Masbaum \cite{HMa}, while the commutativity of the right-hand side diagram is a direct consequence of the definitions.  
\end{proof}

The next technical lemma will be shown in Section \ref{8}.
\begin{lemma}\label{sl4}
Let $L\in SL^h_m(l)$, and $1\leq j\leq i-2\leq m$.  We have $\pi^h_{i,j}(J(L))=0$.
\end{lemma}
We use Lemma \ref{sl4} to establish the following. 
\begin{lemma}\label{spro1}
For $1\le m < p$,  the following diagram commutes
\begin{align*}
\xymatrix{
SL_m(l)\ar[d]_{D^{(p)}} \ar@{->}[r]^{\pi^t_{m}\circ J\textrm{ }\ \ } & (S^{\otimes l})_{m+1}\hbar^m\ar[d]^{\pi^h_m\circ \Delta_{\hbar} ^{(p)}} \\
SL^h_m(pl) \ar@{->}[r]_{\pi^h_{m}\circ J\textrm{ }\ \ } & \langle sl_2\rangle^{(pl)}_{m+1}\hbar^m.
}
\end{align*}
\end{lemma}
\begin{proof}
The diagram in the statement decomposes as 
\begin{align*}
\xymatrix{
SL_m(l)\ar[d]_{D^{(p)}} \ar@{->}[r]^{J\ \ } & J(SL_m(l))\ar[d]_{\Delta_{\hbar} ^{(p)}} \ar@{->}[r]^{\pi^t_{m}} & (S^{\otimes l})_{m+1}\hbar^m\ar[d]^{\pi^h_m\circ \Delta_{\hbar} ^{(p)}} \\
SL^h_m(pl) \ar@{->}[r]_{J\ \ } & J(SL^h_m(pl)) \ar@{->}[r]_{\pi^h_{m}} & \langle sl_2\rangle^{(pl)}_{m+1}\hbar^m,
}
\end{align*}
where the left-hand side square commutes as a general property of the universal $sl_{2}$ invariant.
In order to prove that the right-hand square commutes as well, 
we first show that, given a string link $L\in SL_m(l)$, we have 
\begin{align}\label{eq:JT}
J(L)\in 1+ \bigoplus_{ 1\leq i\leq j\leq m}(S^{\otimes l})_i\hbar^j+(S^{\otimes l})_{m+1}\hbar^m + \hbar^{m+1}U_{\hbar}^{\hat \otimes l}.
\end{align}
In other words, we show that  
\begin{itemize}
\item[\rm{(a)}]  $\pi^t_j(J(L))=0$ for $1\leq j< m-1$,
\item[\rm{(b)}]  $\pi^t_{i,j} (J(L))=0$ for $1\leq j\leq i-2\leq m$. 
\end{itemize}
By Lemma \ref{sl3}, if $\pi_{i,j}^t (J(L))\not =0$ for $i,j\geq 1$, then for $p>i,j,$ we have
\begin{align}\label{ieq}
\pi^h_{i,j} (J({D^{(p)}(L)}))=(\pi^h_{i,j}\circ \Delta_{\hbar}^{(p)} )(J(L))\not =0.
\end{align}
However, as already recalled above, the fact that $L\in SL_m(l)$ implies that $J(D^{(p)}(L))$ is in $SL_m^h(pl)$.  
So (\ref{ieq}) above can neither hold in case $1\leq j=i-1< m-1$ by Theorem \ref{sth1} (ii), which implies (a),
nor in case $1\leq j\leq i-2\leq m$ by  Lemma \ref{sl4}, which implies (b).

Let us now proceed with the proof that the right-hand square of the diagram above is commutative. 
In view of (\ref{eq:JT}), we only need to show the following two claims:  
\begin{itemize}
\item[\rm{(i)}]  $\pi_m^h\circ \Delta_{\hbar} ^{(p)} \big((S^{\otimes l})_i\hbar^j\big)=0$ for any $1\leq i\leq j\leq m$, and 
\item[\rm{(ii)}] $\pi_m^h\circ \Delta_{\hbar} ^{(p)} (\hbar^{m+1}U_{\hbar}^{\hat \otimes l})=0$.
\end{itemize}
Claim (ii) is obvious, from the fact that $\Delta_{\hbar} ^{(p)}(\hbar^{i}U_{\hbar}^{\hat \otimes l})\subset \hbar^{i}U_{\hbar}^{\hat \otimes pl}$ for $i \geq0$.
In order to prove Claim (i), it is enough to show for $0\leq i\leq j$ that
\begin{align*}
 \Delta_{\hbar} ^{(p)} \big((S^{\otimes l})_{i}\hbar^j\big)\subset \prod_{0\leq  u\leq v} (S^{\otimes pl})_{u}\hbar^{v}.
\end{align*}
Recall from \cite{H2} that for $s,n,r\ge 0$, $\Delta_{\hbar} (F^sH^nE^r)$ is equal to 
\begin{align*}
\sum_{0\leq j_1\leq s, 0\leq j_2\leq n, 0\leq j_3\leq r} \begin{bmatrix} s \\ j_1 \end{bmatrix}_q\begin{bmatrix} n \\ j_2 \end{bmatrix}_q
 \begin{pmatrix} r\\ j_3 \end{pmatrix}  F^{s-j_1}H^{n-j_2}K^{j_3}E^{r-j_3}\otimes F^{j_1} K^{s-j_1}H^{j_2} E^{j_3}.
\end{align*}
Since $K=\exp\frac{\hbar H}{2}\in \prod_{t\geq 0}\mathbb{Q} H^t \hbar^t $,
the above formula  implies
\begin{align*}
\Delta_{\hbar} (F^sH^nE^r) \in \prod_{0\leq  t\leq k } (S^{\otimes 2})_{s+n+r+t}\hbar^k.
\end{align*}
Thus we have
\begin{align*}
 \Delta_{\hbar} ^{(p)} \big((S^{\otimes l})_{i}\hbar^j\big)\subset \prod_{0\leq  t\leq k } (S^{\otimes pl})_{i+t}\hbar^{j+k}\subset \prod_{0\leq  u\leq v} (S^{\otimes pl})_{u}\hbar^{v}.
\end{align*}
This concludes the proof of Lemma \ref{spro1}. 
\end{proof}

We can finally proceed with the proof of Theorem \ref{sth2}. 

\begin{proof}[Proof of Theorem \ref{sth2}]
Let $L\in SL_m(l)$ for some $m\ge 1$.  By  (\ref{eq:JT}),  we have 
$$ J^t(L)\in (S^{\otimes l})_{m+1}\hbar^m+ \hbar^{m+1}U_{\hbar}^{\hat \otimes l}, $$
and  thus we only need to prove that $\pi_{m}^t \circ J^t=W\circ \mu_{m+1}$.  
By Lemma \ref{sl3}, it suffices to show that this equality holds after post-composing with $\pi_m^h \circ \bar \Delta^{(p)}$, that is, it suffices to prove that
$$ \pi_m^h \circ \bar \Delta^{(p)} \circ \pi_{m}^t \circ J^t =  \pi_m^h \circ \bar \Delta^{(p)} \circ  W\circ \mu_{m+1}. $$
But according to the commutative diagrams of Lemmas \ref{scom} and \ref{spro1}, this is equivalent to 
proving that 
$$ (\pi_m^h \circ J^h) (D^{(p)}(L) )=(\pi_m^h \circ W \circ  \mu_{m+1})(D^{(p)}(L)),$$
which follows immediately from Theorem \ref{sth2h}. 
Thus we have the assertion.
\end{proof}

\begin{remark}\label{rem:tbw}
We have in particular shown that the universal $sl_2$ invariant for a string link $L\in SL_m(l)$ satisfies (\ref{eq:JT}), 
and the proof of Theorem \ref{sth2} given above relies on the fact that, when applying the projection map $\pi^{t}$ to the above equation, we obtain that 
$J^t(L)\in (S^{\otimes l})_{m+1}\hbar^m+ \hbar^{m+1}U_{\hbar}^{\hat \otimes l}$.
So we could consider an alternative version of the reduction $J^t$ of the universal $sl_2$ invariant for the statement of our main result, by setting 
$$ \tilde{J^t} := \tilde{\pi}^{t}\circ J, $$
where $\tilde{\pi}^{t}$ denotes the quotient map as $\mathbb{Q}$-modules
$$ \tilde{\pi}^{t}\co (S^{\otimes l})[[\hbar]]\to \frac{(S^{\otimes l})[[\hbar]]}{\prod_{ 1\leq i\leq j}(S^{\otimes l})_i\hbar^j}. $$
Clearly, it appears from the above proof, that Theorem \ref{sth2} still holds when replacing $J^t$ with this alternative version $\tilde{J^t}$. 
\end{remark} 

The above observation gives the following, which in particular applies to slice and boundary string links. 
\begin{corollary}\label{cor:final}
Let $L$ be an $l$-component string link with vanishing Milnor invariants. Then we have
 $$J(L)\in 1+ \prod_{ 1\leq i\leq j}(S^{\otimes l})_i\hbar^j. $$  
\end{corollary} 

\section{Universal $sl_2$ invariant and clasper surgery} \label{8}

This section contains the proof of Lemma \ref{sl4} and Proposition \ref{s241}. 
In order to prove these results, we will make use of the theory of claspers, and more precisely we will study the behavior of the universal $sl_2$ invariant under clasper surgery. 

\subsection{A quick review of clasper theory} 
We recall here only the definition and a few properties of claspers for string links, and refer the reader to \cite{H} for more details.

Let $L$ be a string link.  
A \emph{clasper} for $L$ is an embedded surface in $D^2\times [0,1]$, which decomposes into disks and bands, called \emph{edges}, each of which connects two distinct disks.
The disks have either $1$ or $3$ incident edges, and are called {\em leaves} or {\em nodes}, respectively, and the clasper intersects $L$ transversely at a finite number of points, which are all contained in the interiors of the leaves. 
A clasper is called a \textit{tree clasper} if it is connected and simply connected. 
In this paper, we make use of the drawing convention of \cite[Fig. 7]{H} for representing claspers.  

The \emph{degree} of a tree clasper is defined to  the number of  nodes plus 1, i.e., the number of  leaves minus 1.  

Given a clasper $G$ for a string link $L$, we can modify $L$ using the local moves $1$ and $2$ of Figure \ref{fig:clasper} as follows. 
If $G$ contains one or several nodes, pick any leaf of $G$ that is connected to a node by an edge, 
and apply the local move $1$. Keep applying this first move at each node, until none remains: this produces a disjoint union of degree $1$ claspers for the string link $L$ (note indeed that erasing these degree $1$ claspers gives back the string link $L$). Now apply the local move $2$ at each degree $1$ clasper. We say that the resulting string link $L_G$ in $D^2\times [0,1]$ is obtained from $L$ by \emph{surgery along $G$}.  
Note that the isotopy class of $L_G$ does not depend on the order in which the moves were performed. 
\begin{figure}[h!]
\centering
  \includegraphics{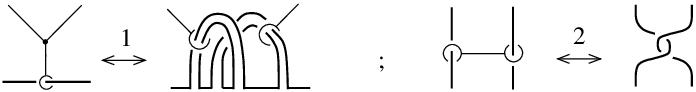}
  \caption{Constructing the image of a string link under clasper surgery.  Here, bold lines represent a bunch of parallel strands from the string link.  }\label{fig:clasper}
\end{figure}

The \emph{$C_k$-equivalence} is the equivalence relation on string links generated by surgeries along tree claspers of degree $k$ and isotopies.  

A clasper for a string link $L$ is called  \emph{simple} if each of its leaves intersects $L$ at one point.  
K. Habiro showed that two string links are $C_k$-equivalent if and only if they are related by surgery along a disjoint union of simple degree $k$ tree claspers. 

In the following, \emph{we will implicitly assume that all tree claspers are simple.}

A tree clasper $G$ for a string link $L$ is called \emph{repeated} if more than one leaf of $G$ intersects the same component of $L$. 
An important property of repeated tree claspers is the following, see for example \cite{FY}.  
\begin{lemma} \label{lem:repeated}
Surgery along a repeated tree clasper preserves the link-homotopy class of (string) links. 
\end{lemma} 

We conclude this subsection with a couple of  standard lemmas in clasper theory.
Proofs are omitted, since they involve the same techniques as in \cite[\S 4]{H}, where similar statements appear.

\begin{lemma} \label{lem:calculus}
Let $C$ be a union of tree claspers for a string link $L$, and let $t$ be a component of $C$ which is a tree clasper of degree $k$.  
Let $C'$ be obtained from $C$ by passing an edge of $t$ across $L$ or across another edge of $C$. 
Then  we have
\begin{align}\label{cal1}
 L_C \stackrel{C_{k+1}}{\sim} L_{C'}.
 \end{align}
 Moreover, if $t$ is repeated,  then the $C_{k+1}$-equivalence in (\ref{cal1}) is realized by surgery along repeated tree claspers.
\end{lemma}
\begin{lemma} \label{lem:slide}
  Let $t_1\cup t_2$ be a disjoint union of a degree $k_1$ and a degree $k_2$ clasper for a string link $L$. 
  Let $t'_1\cup t'_2$ be obtained from $t_1\cup t_2$ by sliding a leaf of $t_1$ across a leaf of $t_2$, as shown below. 
  \begin{figure}[h!]
  \centering
  \includegraphics[scale=0.8]{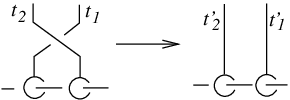}
  \end{figure}
  
  Then we have
 \begin{align}\label{cal2}
 L_{t_1\cup t_2} \stackrel{C_{k_1+k_2}}{\sim} 
 L_{t'_1\cup t'_2}. 
 \end{align}
 Moreover, if one of  $t_1$ and $t_2$ is repeated,  then the $C_{k_1+k_2}$-equivalence in (\ref{cal2}) is realized by surgery along repeated tree claspers.
\end{lemma}
\subsection{Proofs of Lemma \ref{sl4} and Proposition \ref{s241}. }
In this section we prove  Lemma \ref{sl4} and Proposition \ref{s241}.
The proofs rely on two results (Corollary \ref{cor:slide} and Lemma \ref{lem:repeat}) which describe the behavior of the universal $sl_2$ invariant with respect to clasper surgery.

We need an additional technical notion to state these results. 
Recall from Section \ref{2} that  the trivial string link is defined as $\mathbf{1}(=\mathbf{1}_l)=\{p_1,\ldots, p_l\}\times [0,1]$.
We assume that the points $p_i$ are on the line $\{(x, y)\in D^2 \ |\ y=0\}$. 
A tree clasper $T$ for the trivial string link $\mathbf{1}$  is called \emph{overpassing}, if all edges and nodes of $T$ are contained in $\{(x,y)\in D^2 | \ y\leq 0\} \times [0,1]\subset D^2\times [0,1]$. 
In other words, $T$ is overpassing if there is a diagram of $\mathbf{1}\cup T$ which restrict to the standard diagram of $\mathbf{1}$, where the strands do not cross, and where the edges of $T$ overpass $\mathbf{1}$ at all crossings.  

\begin{lemma}\label{slide}
Let  $L$ and $ L_0$ be two link-homotopic $l$-component string links. 
Then for any $m\geq 1$, there exists $n\geq 0$ overpassing repeated tree claspers $R_1,\ldots, R_n$ of degree $\leq m$  for $\mathbf{1}$  such that
\begin{align*}
 L \stackrel{C_{m+1}}{\sim} L_0 \cdot \prod_{j=1}^n \mathbf{1}_{R_j}.
\end{align*}
\end{lemma}

\begin{proof}
By the definition of link-homotopy,  $L$ can be obtained from $L_0$  a finite sequence of self crossing changes, i.e., by surgery along a disjoint union $R$ of $n_1$ repeated degree $1$ tree claspers.

Pick a connected component $R_1$ of $R$. 
By a sequence of crossing changes and leaf slides,  we can ``pull down'' $R_1$ in $D^2\times [0,1]$ so that it leaves in a small neighborhood of $D^2\times \{0\}$, which is disjoint from $R\setminus R_1$ and intersects $L_0$ at $n$ trivial arcs. Apply further crossing changes to ensure that the image $\tilde R_1$ of $R_1$ under this deformation is overpassing. By Lemmas \ref{lem:calculus} and \ref{lem:slide},  we have 
$$ L  \stackrel{C_{2}}{\sim} ( L_0)_{R\setminus R_1}\cdot \mathbf{1}_{\tilde R_1} ,  $$
and the $C_{2}$-equivalence is realized by surgery along repeated tree claspers of degree $2$.
Applying this procedure to each of the $n_1$ connected components of $R$ successively, we eventually obtain that 
$$ L =  \left( L_0\cdot \prod_{1\leq i\leq n_1} \mathbf{1}_{\tilde R_i} \right)_ {R^{(2)}},  $$
where each $\tilde R_i$ is an overpassing tree clasper of degree $1$, 
and $R^{(2)}$ is a disjoint union of repeated  tree claspers of degree $2$. 

Next, we apply the same ``pull down'' procedure to each connected component of  $R^{(2)}$ successively. 
Using the same lemmas, we then have that 
 $$ L = \left(L_0\cdot  \prod_{1\leq i_2\leq n_2} \mathbf{1}_{\tilde R^{(2)}_{i_2}} \right)_{R^{(3)}},  $$
where each $R^{(2)}_i$ is an overpassing repeated  tree clasper of degree at most $2$, 
and $R^{(3)}$ is a disjoint union of  repeated tree claspers of  degree $\geq 3$.

Iterating this procedure, we obtain that, for any integer $m\ge 1$, we have  
 $$ L = \left( L_0\cdot \prod_{1\leq i_m\leq n_m} \mathbf{1}_{\tilde R^{(m)}_{i_m}} \right)_{R^{(m+1)}},  $$
where $\tilde R^{(m)}_i$ is an overpassing  repeated tree clasper of degree at most $m$, 
and $R^{(m+1)}$ is a disjoint union of repeated tree claspers of  degree $\geq m+1$.

This completes the proof.  
\end{proof}

Since the universal $sl_2$ invariant modulo the ideal $\hbar^{k}U_{\hbar}^{\hat \otimes l}$ is a finite type invariant of degree $<k$, 
hence is an invariant of $C_{k}$-equivalence \cite{H}, the multiplicativity of the universal $sl_2$ invariant implies the following.
\begin{corollary}\label{cor:slide}
Let  $L$ and $ L_0$ be two link-homotopic $l$-component string links. 
Then for any $m\geq 1$, there exists $n\geq 0$ overpassing repeated tree claspers $R_1,\ldots, R_n$ of degree $\leq m$  for $\mathbf{1}$  such that
\begin{align*}
 J(L)\equiv J(L_0) \cdot \prod_{j=1}^n J(\mathbf{1}_{R_j}) \oh{m+1}.
\end{align*}
\end{corollary}

We will apply this result to the case where $L_0$ is the explicit representative for the link-homotopy class of $L$ given in Lemma \ref{lem:braidlh}, whose universal $sl_{2}$ invariant was studied in details in Section \ref{6}. 
By Corollary \ref{cor:slide}, we are thus lead to studying the universal $sl_{2}$ invariant of string links obtained from $\mathbf{1}$ by surgery along an overpassing repeated tree clasper: this is the subject of Lemma \ref{lem:repeat} below. 

For $x=F^{s_1}H^{n_1}E^{r_1}\otimes F^{s_2}H^{n_2}E^{r_2}\otimes \cdots \otimes F^{s_l}H^{n_l}E^{r_l}\in S^{\otimes l}$,
set
 \begin{align*}
\mathrm{supp}(x)= \sharp \{ 1\leq i\leq l\ | \ s_i+n_i+r_i\not =0\}, 
 \end{align*}
that is, roughly speaking, the number of nontrivial tensorands. 
We denote by $(S^{\otimes l})_{\mathrm{supp}\le n}$ the $\mathbb{Q}$-submodule of $S^{\otimes l}$ spanned by all monomials $x$ such that $\mathrm{supp}(x)\leq n$.

\begin{lemma}\label{lem:repeat}
Let $C$ be an overpassing repeated tree clasper for $\mathbf{1}\in SL(l)$.
We have
\begin{align*}
J(\mathbf{1}_C)\in 1+\prod_{j\geq 1} (S^{\otimes l})_{\mathrm{supp}\le j}\hbar^j.
\end{align*}
\end{lemma}

\begin{proof}
Since $C$ is an overpassing repeated tree clasper for $\mathbf{1}_l$, there exists an $l$-component  braid $B$ such that
\begin{align}\label{over}
(\mathbf{1}_l)_C=B \cdot  ((\mathbf{1}_k)_{C'} \otimes \mathbf{1}_{l-k})\cdot B^{-1}, 
\end{align} 
where $k$ denotes the number of strands of $\mathbf{1}_l$ intersecting $C$, and where $C'$ denotes the image of $C$ under this isotopy. 
(Recall that $\otimes$ denotes the horizontal juxtaposition of string links.)

Let $m$ denote the degree of $C$ (and $C'$). 
Since $J((\mathbf{1}_k)_{C'})\equiv 1 \oh{m}$ and $k\leq$ $\sharp \{$leaves of $C\}-1 = m$, we have
\begin{align*}
J((\mathbf{1}_k)_{C'}) &\in 1 + \prod_{j\geq m}S^{\otimes k}\hbar^j
\\
&\subset 1+ \prod_{j\geq k}S^{\otimes k}\hbar^j,
\end{align*}
which implies that 
\begin{align*}
J((\mathbf{1}_k)_{C'} \otimes \mathbf{1}_{l-k})& \in 1 + \prod_{j\geq k}(S^{\otimes l})_{\mathrm{supp}\le k}\hbar^j.
\end{align*}

So, by Equation (\ref{over}), in order to obtain the desired result it only remains to show that $\prod_{j\geq 1} (S^{\otimes l})_{\mathrm{supp}\le j}\hbar^j$
is invariant under the braid group action.
Here, the braid group acts on $U_{\hbar}^{\hat \otimes l}$ by quantized permutation: the action of Artin generator $\sigma_n$
on an element $x$ is given by $R^{(l)}_{n+1,n} (\bar \sigma_n (x)) (R^{-1})^{(l)}_{n+1,n}$, 
where $\bar \sigma_n(x)$ denotes the permutation of the $n$th and $(n+1)$th tensorands of $x$. 

Hence it suffices to prove that, for any monomial $x=(x_1\otimes \cdots \otimes x_k\otimes 1^{\otimes l-k})\hbar^j$ with  
$x_1,\ldots,x_k\in S$, $j\geq k$, and any $n \in \{1, . . . ,l-1\}$, we have
 \begin{align*}
 R^{(l)}_{n+1,n}.x.(R^{-1})^{(l)}_{n+1,n}\in \prod_{j\geq 1}(S^{\otimes l})_{\mathrm{supp}\le j}\hbar^j,
 \intertext{and}
 (R^{-1})^{(l)}_{n,n+1}.x.R^{(l)}_{n,n+1}\in \prod_{j\geq 1}(S^{\otimes l})_{\mathrm{supp}\le j}\hbar^j.
 \end{align*}
 We prove the first inclusion. The second one is similar.
This is clear when $1\leq n\leq k-1$ and  $k+1 \leq n$. When $n=k$, since $R^{\pm1}\equiv 1 \oh{}$, we have 
\begin{align*}
 R^{(l)}_{k+1,k}.x.(R^{-1})^{(l)}_{k+1,k}&= \Big(x_1\otimes \cdots \otimes R_{21}(x_k\otimes 1)(R^{-1})_{21}\otimes 1^{\otimes l-k-1} \Big)\hbar^j
\\
&\in ( x_1\otimes \cdots \otimes x_k \otimes 1^{\otimes l-k})\hbar^j  + \prod_{i\geq j+1}(S^{\otimes l})_{\mathrm{supp}\le k+1}\hbar^i
\\
&\subset ( x_1\otimes \cdots \otimes x_k \otimes 1^{\otimes l-k})\hbar^j  + \prod_{i\geq j+1}(S^{\otimes l})_{\mathrm{supp}\le j+1}\hbar^i.
\end{align*}
This completes the proof.
\end{proof}

\begin{corollary}\label{coslide}
If $L\in SL^h_m(l)$,  then  we have
\begin{align}\label{eq:JT2}
J(L) & \in  1 +  (W\circ \mu^h_m)(L) +  \bigoplus_{j=1}^m  (S^{\otimes l})_{\mathrm{supp}\le j}\hbar^j  + \hbar^{m+1}U_{\hbar}^{\hat \otimes l}.
\end{align}
\end{corollary}
\begin{proof}
Recall from Lemma \ref{lem:braidlh} that $L\in SL_m^h(l)$ is link-homotopic to $b=b_m^Lb_{m+1}^L\cdots b_{l-1}^L$.
By  Proposition \ref{sp} we have
\begin{align}\label{eq:Jbraid}
\begin{split}
 J(b)  &\equiv 1 + \left( \sum_{J\in \mathcal{I}_{m+1}} \mu_J(L).\varsigma^{(l)}_J\right) \hbar^m 
 \\
&\equiv 1 +  (W\circ \mu^h_m) (L) \oh{m+1},
\end{split}
\end{align}
where the second equality follows from Equation (\ref{equivalence}). 
Since $\prod_{j}  (S^{\otimes l})_{\mathrm{supp}\le j} \hbar^j$ is closed under multiplication, Corollary \ref{cor:slide} and Lemma  \ref{lem:repeat} imply that 
\begin{align*}
J(L) &\in  J(b) \cdot \left( 1+\prod_{j}  (S^{\otimes l})_{\mathrm{supp}\le j} \hbar^j \right)  + \hbar^{m+1}U_{\hbar}^{\hat \otimes l}
\\
&\subset \left(1 +  (W\circ \mu^h_m) (L)\right)\cdot \left( 1+\prod_{j}  (S^{\otimes l})_{\mathrm{supp}\le j} \hbar^j \right)  + \hbar^{m+1}U_{\hbar}^{\hat \otimes l}
\\
&\subset 1 +  (W\circ \mu^h_m)(L)  +  \bigoplus_{j=1}^m  (S^{\otimes l})_{\mathrm{supp}\le j}\hbar^j  + \hbar^{m+1}U_{\hbar}^{\hat \otimes l}.
\end{align*}
This completes the proof.
\end{proof}

We can now prove Lemma \ref{sl4} and Proposition \ref{s241}. 

\begin{proof}[Proof of Lemma \ref{sl4}]
Let $L\in SL^h_m(l)$ and $1\leq j\leq i-2\leq m$.  
By Corollary \ref{coslide}, we have
\begin{align*}
\pi^h_{i,j }(J(L))&\in \pi^h_{i,j }\left(1 +  (W\circ \mu^h_m)(L) +  \bigoplus_{k=1}^m  (S^{\otimes l})_{\mathrm{supp}\le k}\hbar^k  + \hbar^{m+1}U_{\hbar}^{\hat \otimes l}\right)
\\
&=\pi^h_{i,j }\left((W\circ \mu^h_m)(L)\right) +  \pi^h_{i,j }\left(\bigoplus_{k=1}^m  (S^{\otimes l})_{\mathrm{supp}\le k}\hbar^k\right).
\end{align*}
But the right  hand side is equal to $0$ since we have
\begin{align*}
(W\circ \mu^h_m)(L)&\in  \langle sl_2 \rangle ^{(l)}_{ m+1}\hbar^m,
\intertext{and}
\bigoplus_{k=1}^m  (S^{\otimes l})_{\mathrm{supp}\le k}\hbar^k &\cap \langle sl_2 \rangle ^{(l)}_{i}\hbar ^j=\emptyset,
\end{align*}
since for any monomial $x\in \langle sl_2 \rangle ^{(l)}_{ i}$ we have $\mathrm{supp}(x)=i\geq j+2$.

This completes the proof.
\end{proof}
\begin{proof}[Proof of Proposition \ref{s241}]
Let $L, L' \in SL^h_m(l)$  be two link-homotopic  string links. 
By Corollary \ref{cor:slide} and Lemma  \ref{lem:repeat}, 
together with the fact that $\prod_{j}  (S^{\otimes l})_{\mathrm{supp}\le j} \hbar^j$ is closed under multiplication, we have 
\begin{align*}
J(L) &\in J({ L'}) \cdot \left( 1+\prod_{j}  (S^{\otimes l})_{\mathrm{supp}\le j} \hbar^j \right)  +\hbar^{m+1}U_{\hbar}^{\hat \otimes l}.  
\end{align*}
Then  Corollary  \ref{coslide} implies  that $J^h(L) \equiv  J^h(L') \oh{m+1}$, as desired. 
\end{proof}

In a similar spirit as Remark \ref{rem:tbw}, we have the following.
\begin{remark}\label{rem:tbw2}
By Corollary \ref{coslide}, the universal $sl_2$ invariant for a string link $L$ in $SL^h_m(l)$ satisfies (\ref{eq:JT2}). 
So we have a variant of Theorem \ref{sth2h}, using an alternative version of the reduction $J^h$ of the universal $sl_2$ invariant $J$, by setting 
$$ \tilde{J^h} := \tilde{\pi}^{h}\circ J, $$
where $\tilde{\pi}^{h}$ denotes the quotient map as $\mathbb{Q}$-modules
$$ \tilde{\pi}^{h}\co (S^{\otimes l})[[\hbar]]\to \frac{(S^{\otimes l})[[\hbar]]}{\prod_j(S^{\otimes l})_{\mathrm{supp}\le j}\hbar^j }. $$
Indeed, it follows immediately from Corollary \ref{coslide} that if $L\in SL^h_m(l)$, then 
$$\tilde{J^h}(L)\equiv (W\circ \mu^h_m)(L)  \oh{m+1}.$$ 
\end{remark} 
In particular, we obtain the following. 
\begin{corollary}\label{cor:final2}
Let $L$ be a link-homotopically trivial $l$-component string link. Then we have
$$J(L)\in 1+ \bigoplus_{j=1}^{l-1} (S^{\otimes l})_{supp\le j}\hbar^j+ \hbar^lU_{\hbar}^{\hat \otimes l}. $$ 
\end{corollary}

\section*{Acknowledgments}
The first author is supported by the French ANR research projects ``VasKho'' ANR-11-JS01-00201 and ``ModGroup'' ANR-11-BS01-02001. 
The second author is supported by JSPS KAKENHI Grant Number 15K17539. 
The authors wish to thank Kazuo Habiro and Marco Castronovo for helpful comments and conversations.
They also thank the referee for many useful remarks and comments.


\begin{thebibliography}{0}


\bibitem{BN} D. Bar-Natan, \emph{On the Vassiliev knot invariants}, Topology \textbf{34} (1995), 423--472.

\bibitem{BN2} D. Bar-Natan, {\it Vassiliev homotopy string link invariants}, J. Knot Theory Ram. {\bf 4}, no. 1 (1995), 13--32.

\bibitem{wheels} D. Bar-Natan, S. Garoufalidis, L. Rozansky, D.P. Thurston, \emph{Wheels, wheeling, and the Kontsevich integral of the unknot}, Israel J. Math. \textbf{119} (2000), 217--237.

\bibitem{casson} A.J. Casson, \emph{Link cobordism and Milnor's invariant}, Bull. London Math. Soc. \textbf{7} (1975), 39--40.

\bibitem{CD} S.V. Chmutov and S.V. Duzhin, \emph{The Kontsevich Integral}, Acta Appl. Math. \textbf{66} (2000), 155--190.

\bibitem{CDM} S. Chmutov, S. Duzhin and J. Mostovoy, \emph{Introduction to Vassiliev knot invariants}, Cambridge University Press, Cambridge, 2012. 

\bibitem{FY} T. Fleming and A. Yasuhara, {\it Milnor's invariants and self $C_k$-equivalence}, Proc. Amer. Math. Soc. {\bf 137} (2009) 761-770. 

\bibitem{HL1} N. Habegger, X.S. Lin, \emph{The classification of links up to link-homotopy}, 
            J. Amer. Math. Soc. \textbf{3} (1990), no. 2, 389-419.
%
\bibitem{HL2} N. Habegger, X.S. Lin, \emph{On Link Concordance and Milnor's $\overline \mu$ Invariants}, Bull. London Math. Soc. \textbf{30} (1998), 419-428.
%
\bibitem{HMa} N. Habegger, G. Masbaum, {\it The Kontsevich integral and Milnor's invariants}, Topology {\bf 39} (2000), no. 6, 1253--1289.
%
\bibitem{H} K. Habiro, {\it Claspers and finite type invariants of links}, Geom. Topol. {\bf 4} (2000), 1--83.

\bibitem{H1} {K. Habiro}, {\it Bottom tangles and universal invariants}, {Alg. Geom. Topol. \textbf{6} (2006), 1113--1214. }

\bibitem{H2} {K. Habiro}, {\it A unified Witten-Reshetikhin-Turaev invariants for integral homology spheres}. {Invent. Math. \textbf{171} (2008), no. 1, 1--81. }

\bibitem{Kassel} C. Kassel, {\it Quantum groups}, 
Graduate Texts in Mathematics \textbf{155}, Springer-Verlag, New York, 1995. 

\bibitem{Ko} M. Kontsevich, \emph{Vassiliev's knot invariants},  ``I. M. Gel'fand Seminar'', 137--150, Adv. Soviet Math., 16, Part 2,
  Amer. Math. Soc., Providence, RI, 1993.
  
\bibitem{Law1} R. J. Lawrence, {\it A universal link invariant}, in ``The Interface of Mathematics and Particle Physics'' (Oxford, 1988), 
Inst. Math. Appl. Conf. Ser. New Ser., vol. \textbf{24}, Oxford University Press, New York, 1990, 151--156. 
%
\bibitem{Law2} R. J. Lawrence, {\it A universal link invariant using quantum groups}, in ``Differential Geometric Methods in Theoretical Physics'' 
(Chester, 1989), World Sci. Pub. (1989), 55--63.
%
\bibitem{LM} T.T.Q. Le and J. Murakami, {\it The universal Vassiliev-Kontsevich invariant for framed oriented links}, Compos. Math. {\bf 102} (1996), 41--64. 
%
\bibitem{Lin} X.S. Lin, {\it Power series expansions and invariants of
  links}, in ``Geometric topology'', AMS/IP Stud. Adv. Math. 2.1,
  Amer. Math. Soc. Providence, RI (1997) 184--202.
  
\bibitem{MS} J.B. Meilhan, S. Suzuki, \emph{Riordan trees and the homotopy $sl_2$ weight system}, J. Pure Appl. Algebra (2016), http://dx.doi.org/10.1016/j.jpaa.2016.07.012.

\bibitem{MYpjm} J.B. Meilhan, A. Yasuhara, \emph{On Cn-moves for links}, Pacific J. Math. \textbf{238} (2008), 119--143. 

\bibitem{MY} J.B. Meilhan, A. Yasuhara, {\it Milnor invariants and the HOMFLYPT polynomial}, Geom. Topol {\bf 16} (2012), 889--917
%
\bibitem{Milnor} J. Milnor, {\it Link groups}, Ann. of Math. (2) {\bf 59} (1954), 177--195.
%
\bibitem{Milnor2} J. Milnor, {\it Isotopy of links}, Algebraic geometry and topology, 
      A symposium in honor of S. Lefschetz, pp. 280--306, 
      Princeton University Press, Princeton, N. J., 1957. 

\bibitem{Polyak} M. Polyak, \emph{Skein relations for Milnor’s $\mu$-invariants}, Algebr. Geom. Topol.  \textbf{5} (2005), 1471--1479. 

\bibitem{O} T. Ohtsuki, {\it Colored ribbon Hopf algebras and universal invariants of framed links} J. Knot Theory Ram. \textbf{2} (1993), no. 2, 211--232.

\bibitem{ohtsuki} T. Ohtsuki, \emph{Quantum Invariants, A Study of Knots, 3-Manifolds, and their Sets}, World Scientific Publishing Company.
%
\bibitem{RT} N. Y. Reshetikhin and V. G. Turaev, \emph{Ribbon graphs and their invariants derived from quantum groups},  Comm. Math. Phys. \textbf{127} (1990), no. 1, 1--26.
%
\bibitem{rozansky} L. Rozansky, \emph{Reshetikhin's formula for the Jones polynomial of a link: Feynman diagrams and Milnor's linking numbers}, 
 in ``Topology and physics'', J. Math. Phys. \textbf{35} (1994), no. 10, 5219--5246. 
%
\bibitem{stallings} J. Stallings, \emph{Homology and central series of groups, J. Algebra}, {\bf 2}(1965), 170--181. 

\bibitem{sakie1} {S. Suzuki}, {\it On the universal $sl_2$ invariant of ribbon bottom tangles.} {Algebr. Geom. Topol.  \textbf{10} (2010), no. 2, 1027--1061. }

\bibitem{sakie2} {S. Suzuki}, {\it On the universal $sl_2$ invariant of boundary bottom tangles}, Algebr. Geom. Topol.  \textbf{12} (2012), 997--1057.

\bibitem{sakie3} {S. Suzuki}, {\it On the universal $sl_2$ invariant of Brunnian bottom tangles}, Math. Proc. Camb. Phil. Soc.  \textbf{154} (2013), no.1, 127--143.

\bibitem{yasuhara} A. Yasuhara, \emph{Self Delta-equivalence for Links Whose Milnor's Isotopy Invariants Vanish}, Trans. Amer. Math. Soc. 
\textbf{361} (2009), 4721--4749.  

\end{thebibliography}
\end{document}